\documentclass[12pt]{article}

\usepackage{amsmath,amssymb,bm,caption,color,graphicx,natbib,amsthm,subcaption}
\usepackage{fullpage}
\usepackage{lscape}

\newtheorem{theorem}{Theorem}
\newtheorem{lemma}{Lemma}

\renewcommand{\d}[1]{\ensuremath{\operatorname{d}\!{#1}}}

\usepackage[usenames,dvipsnames]{xcolor}

\begin{document}

\title{Convergence Rates for a Class of Estimators Based on Stein's Method}
\author{Chris J. Oates$^{1,4}$, Jon Cockayne$^2$, Fran\c{c}ois-Xavier Briol$^2$, Mark Girolami$^{3,4}$\\
$^1$Newcastle University \\
$^2$University of Warwick \\
$^3$Imperial College London\\
$^4$Alan Turing Institute}

\maketitle

\begin{abstract}
Gradient information on the sampling distribution can be used to reduce the variance of Monte Carlo estimators via Stein's method.
An important application is that of estimating an expectation of a test function along the sample path of a Markov chain, where gradient information enables convergence rate improvement at the cost of a linear system which must be solved.
The contribution of this paper is to establish theoretical bounds on convergence rates for a class of estimators based on Stein's method.
Our analysis accounts for (i) the degree of smoothness of the sampling distribution and test function, (ii) the dimension of the state space, and (iii) the case of non-independent samples arising from a Markov chain.
These results provide insight into the rapid convergence of gradient-based estimators observed for low-dimensional problems, as well as clarifying a curse-of-dimension that appears inherent to such methods.
\end{abstract}


\section{Introduction} \label{intro}

This paper considers methods to estimate the integral
$$
\int f \d\Pi
$$ 
of a test function $f$ against a distribution $\Pi$ based on evaluation of $f$ at a finite number $n$ of inputs. Our work is motivated by challenging settings in which (i) the variance $\sigma^2(f) = \int (f - \int f \d\Pi )^2 \d\Pi$ is large relative to $n$, and (ii) the distribution $\Pi$ is only available up to an unknown normalisation constant.
Such problems arise in Bayesian statistics when the cost of sampling from the posterior is prohibitive, requiring that posterior expectations are approximated based on a small number $n$ of evaluations of the integrand.
Indeed, the intrinsic accuracy of ergodic averages, such as obtained via Markov chain Monte Carlo (MCMC) methods \citep{Robert}, can lead to unacceptably high integration error when $n$ is small.
This paper considers a class of estimators inspired by Stein's method \citep{Stein}, based on integration-by-parts in this context: 
\begin{eqnarray}
\int f \d\Pi & = & - \int \left( \int f \d x \right) \cdot \frac{\mathrm{d}}{\mathrm{d}x} \log \pi \d\Pi, \label{Stein}
\end{eqnarray}
subject to boundary conditions, where $\pi$ is a density for $\Pi$.
These estimators ensure an integration error $o_{\mathrm{P}}(n^{-\frac{1}{2}})$, provided that gradient information on the sampling distribution can be obtained.
This is often the case; indeed, sophisticated software for automatic differentiation of statistical models has been developed \citep[e.g.][]{Carpenter,Maclaurin}.

\vspace{5pt}
\noindent {\bf Main Contribution:}
The primary contribution of this paper is to establish convergence rates for a class of estimators based on Stein's method.
These estimators, first described in \cite{Oates}, require as input both function evaluations $\{f(\bm{x}_i)\}_{i=1}^n$ and gradient evaluations $\{\nabla \log \pi(\bm{x}_i)\}_{i=1}^n$, where the states $\{\bm{x}_i\}_{i=1}^n$ themselves can be either independent or correlated draws from $\Pi$.
Our central results are asymptotic rates for integration error; these enable us to compare and quantify the improvement in estimator precision relative to standard Monte Carlo methods and in doing so we fill a theoretical void.

The estimators that we consider can be viewed as a control variate (or `control functional') method, and this concept is discussed next.

\vspace{5pt}
\noindent {\bf Control Functionals:}
The classical control variate method proceeds by seeking a collection of non-trivial statistics $\{\psi_i\}_{i=1}^k$, such that each satisfies $\int \psi_i \d\Pi = 0$.
Then a surrogate function 
$$
f' = f - a_1\psi_1 - \dots - a_k\psi_k
$$ 
is constructed such that automatically $\int f' \d\Pi = \int f \d\Pi$ and, for suitably chosen $\{a_i\}_{i=1}^k$, a variance reduction $\sigma^2(f') < \sigma^2(f)$ might be obtained; for further details see e.g. \cite{Rubinstein}.
For specific problems it is sometimes possible to identify control variates, for example based on physical considerations \citep[e.g.][]{Assaraf2}.
For Monte Carlo integration based on Markov chains, it is sometimes possible to construct control variates based on statistics relating to the sample path.
In this direction, the problem of constructing control variates for discrete state spaces was essentially solved by \cite{Andradottir} and for continuous state spaces, recent contributions include \cite{Hammer,Dellaportas,Li,Mijatovic,Mijatovic2}.
Control variates can alternatively be constructed based on gradient information on the sampling distribution \citep{Assaraf,Mira,Oates}.

The estimators considered here stem from a recent development that extends control variates to control {\it functionals}.
This idea is motivated by the observation that the methods listed above are (in effect) solving a misspecified regression problem, since in general $f$ does not belong to the linear span of the statistics $\{\psi_i\}_{i=1}^k$.
The recent work by \cite{Mijatovic,Oates} alleviates model misspecification by increasing the number $k$ of statistics alongside the number $n$ of samples so that the limiting space spanned by the statistics $\{\psi_i\}_{i=1}^\infty$ is dense in a class of functions that contains the test function $f$ of interest.
Both methods provide a non-parametric alternative to classical control variates whose error is $o_{\mathrm{P}}(n^{-\frac{1}{2}})$.
Of these two proposed solutions, \cite{Mijatovic} is not considered here since it is unclear how to proceed when $\Pi$ is known only up to a normalisation constant.
On the other hand the control functional method of \cite{Oates} is straight-forward to implement when gradients $\{\nabla \log \pi(\bm{x}_i)\}_{i=1}^n$ are provided.
Understanding the theoretical properties of this method is the focus of the present research.

\vspace{5pt}
\noindent {\bf Technical Contribution:}
This paper establishes that the estimators of \cite{Oates} incur an integration error $O_{\mathrm{P}}(n^{-\frac{1}{2} - \frac{a \wedge b}{d} + \epsilon})$, where $a$ is related to the smoothness of the density $\pi$, $b$ is related to the smoothness of the test function $f$, $d$ is the dimension of the domain of integration and $\epsilon > 0$ can be arbitrarily small (a notational convention used to hide logarithmic factors).
This analysis provides important insight into the strong performance that has been observed for these estimators in certain low-dimensional applications \citep{Oates,Liu4}.
Indeed, recall that the (na\"{i}ve) computational cost associated with these methods, i.e. the cost of solving a linear system, is $c = O(n^3)$.
This cost can also involve a large constant factor when hyper-parameters are to be jointly estimated.
Thus, whilst for standard Monte Carlo methods an estimator error of $O_{\mathrm{P}}(c^{-\frac{1}{2}})$ can be achieved at computational cost $c$, for gradient-based control functionals 
$$
\text{error for cost }c \quad = \quad O_{\mathrm{P}}\left( (c^{\frac{1}{3}})^{-\frac{1}{2} - \frac{a \wedge b}{d} + \epsilon} \right) \quad = \quad O_{\mathrm{P}} \left( c^{-\frac{1}{6} + \frac{d - a \wedge b}{3d} + \epsilon } \right) .
$$
This demonstrates that gradient-based control functionals have asymptotically lower error for the same fixed computational cost $c$ whenever $a \wedge b > d$, which occurs when both the density $\pi$ and the test function $f$ are sufficiently smooth.
In the situation where the computational bottleneck is evaluation of $f$, not solution of the linear system, then the computational gain can be even more substantial.
At the same time, the critical dependence on $d$ highlights the curse-of-dimension that appears inherent to such methods.
Going forward, these results provide a benchmark for future high-dimensional development.

\vspace{5pt}
\noindent {\bf Relation to Other Acceleration Methods:}
Accelerated rates of convergence can be achieved by other means, including quasi-Monte Carlo \citep[QMC;][]{Niederreiter}.
Consider the ratio estimator:
\begin{eqnarray}
\int f \d\Pi & \approx & \frac{ \frac{1}{n} \sum_{i=1}^n f(\bm{x}_i) \pi(\bm{x}_i) }{ \frac{1}{n} \sum_{i=1}^n \pi(\bm{x}_i) } \label{ratio estimator}
\end{eqnarray}
For appropriate randomised point sets $\{\bm{x}_i\}_{i=1}^n$, the ratio estimator converges at a rate limited by the least smooth of $f \cdot \pi$ and $f$, i.e. limited by $\frac{a \wedge b}{d}$ (at least, in the absence of additional conditions on the mixed partial derivatives, which we have not assumed)\footnote{In this section the notation $a$ and $b$ is used as a shorthand for the ``smoothness'' of, respectively, $\pi$ and $f$.
The precise mathematical definition of $a$ and $b$ differs between manuscripts and the results discussed here should not be directly compared.}.
See \cite{Dick} for a recent study of this approach in the context of Bayesian inference for an unknown parameter in a partial differential equation model.

The method studied herein can be contrasted with QMC methods in at least two respects:
(1) The states $\{\bm{x}_i\}_{i=1}^n$ can be independent (or correlated) draws from $\Pi$, which avoids the need to specifically construct a point set.
This is an important benefit in cases where the domain of integration is complicated - indeed, our results hold for any domain of integration for which an interior cone condition can be established.
(2) The estimator studied herein is unbiased, whereas ratio estimators of the form in Eqn. \ref{ratio estimator} will be biased in general.
The unbiased nature of the estimator, in common with standard Monte Carlo methods, facilitates convenient diagnostics to estimate the extent of Monte Carlo error and is therefore useful.

Recent work from \cite{Delyon} and \cite{Azais} considered estimators of the form
\begin{eqnarray}
\int f \d\Lambda & \approx & \frac{1}{n} \sum_{i=1}^n \frac{f(\bm{x}_i)}{\hat{\pi}(\bm{x}_i)} \label{Delyon estimator}
\end{eqnarray}
where $\hat{\pi}$ is a kernel density estimate for $\pi = \mathrm{d}\Pi / \mathrm{d}\Lambda$ based on a collection of (possibly correlated) draws $\{\bm{x}_i\}_{i=1}^n$ from $\Pi$.
Again, theoretical results established an error of $o_{\mathrm{P}}(n^{-\frac{1}{2}})$ with an explicit rate gated by a term of the form $\frac{a \wedge b}{d}$.
However, this approach applies to integrals with respect to a known, normalised reference measure $\Lambda$ rather than with respect to $\Pi$.

\vspace{5pt}
\noindent {\bf Outline:}
Below in Sec. \ref{methods} we describe the class of estimators that were considered and present our main theoretical results, including the case of non-independent samples arising from a Markov chain sample path.
Our theoretical analysis combines error bounds from the scattered data approximation literature with stability results for Markov chains; proofs are contained in the electronic supplement.
Numerical results in Sec. \ref{illustration} confirm these error rates are realised.
Finally the importance of our findings is discussed in Sec. \ref{discuss}.

\section{Methods} \label{methods}

First we fix notation before describing the estimation method. 

\subsection{Set-up and Notation} \label{setup}

Consider an open and bounded set $\mathcal{X} \subset \mathbb{R}^d$, $d \in \mathbb{N}$, with boundary $\partial \mathcal{X}$.
Let $\mathcal{B} = \mathcal{B}(\mathcal{X} \cup \partial \mathcal{X})$ denote the Borel $\sigma$-algebra on $\mathcal{X} \cup \partial \mathcal{X}$ and equip $(\mathcal{X} \cup \partial \mathcal{X},\mathcal{B})$ with the reference measure $\Lambda$ induced from the restriction of Lebesgue measure on $\mathbb{R}^d$.
Further, consider a random variable $\bm{X}$ on $\mathcal{X} \cup \partial \mathcal{X}$ with distribution $\Pi$ and suppose $\Pi$ admits a density $\pi = \d\Pi / \d\Lambda$. 

The following notation will be used: $\mathbb{N}_0 := \mathbb{N} \cup \{0\}$, $a \wedge b := \min(a,b)$, $a_+ := \max(a,0)$, $\bm{1} = [1,\dots,1]^\top$, $\|\bm{x}\|_2^2 := \sum_{i=1}^d x_i^2$, $\nabla_{\bm{x}} : = [\partial/\partial x_1,\dots,\partial/\partial x_d]^\top$, $1_A(\bm{x}) = 1$ is the indicator of the event $\bm{x} \in A$.
Write $L^2(\mathcal{X},\Pi)$ for the vector space of measurable functions $f:\mathcal{X} \rightarrow \mathbb{R}$ for which $\sigma^2(f) := \int (f - \int f \d\Pi)^2 \d\Pi$ exists and is finite. 
Write $C^k(\mathcal{X})$ for the set of measurable functions for which continuous partial derivatives exist on $\mathcal{X}$ up to order $k \in \mathbb{N}_0$.
A function $g:\mathcal{X} \times \mathcal{X} \rightarrow \mathbb{R}$ is said to be in $C_2^k(\mathcal{X})$ if $\partial^{2k}g / \partial x_{i_1} \dots \partial x_{i_k} \partial x_{j_1}' \dots \partial x_{j_k}'$ is $C^0(\mathcal{X} \times \mathcal{X})$ for all $i_1,\dots,i_k,j_1,\dots,j_k \in \{1,\dots,d\}$.
The notation $\|f\|_\infty := \sup_{\bm{x} \in \mathcal{X}} |f(\bm{x})|$ will be used.

\subsection{Control Functionals}

This section introduces the control functional method for integration, a non-parametric extension of the classical control variate method.
Recall that the trade-off between random sampling and deterministic approximation in the context of integration is well-understood \citep{Bakhvalov}.
Our starting point is, in a similar vein, to establish a trade-off between random sampling and {\it stochastic} approximation.

We assume throughout that the test function $f$ belongs to $L^2(\mathcal{X},\Pi)$ and that the boundary $\partial\mathcal{X}$ is piecewise smooth.
Consider an independent sample from $\Pi$, denoted $\mathcal{D} = \{\bm{x}_i\}_{i=1}^n$.
This is partitioned into disjoint subsets $\mathcal{D}_0 = \{\bm{x}_i\}_{i=1}^m$ and $\mathcal{D}_1 = \{\bm{x}_i\}_{i=m+1}^n$, where $1 \leq m < n$.
Although $m$, $n$ are fixed, we will be interested in the asymptotic regime where $m = O(n^\gamma)$ for some $\gamma \in [0,1]$.
Consider constructing an approximation $f_m \in L^2(\mathcal{X},\Pi)$ to $f$, based on $\mathcal{D}_0$.
Stochasticity in $f_m$ is induced via the sampling distribution of elements in $\mathcal{D}_0$.
The integral $\int f_m \d\Pi$ is required to be analytically tractable; we will return to this point.

The estimators that we study take the form
\begin{eqnarray}
I_{m,n} := \frac{1}{n-m}\sum_{i=m+1}^{n} f(\bm{x}_{i}) - \left( f_m(\bm{x}_i) - \int f_m \d\Pi \right) . \label{splitting estimators}
\end{eqnarray}
Such sample-splitting estimators are unbiased, i.e. $\mathbb{E}_{\mathcal{D}_1}[I_{m,n}] = \int f \d\Pi$, where the expectation here is with respect to the sampling distribution $\Pi$ of the $n-m$ random variables that constitute $\mathcal{D}_1$, and is conditional on fixed $\mathcal{D}_0$.
The corresponding estimator variance, again conditional on $\mathcal{D}_0$, is $\mathbb{V}_{\mathcal{D}_1}[I_{m,n}] = (n-m)^{-1} \sigma^2(f - f_m)$.
This formulation encompasses control variates as a special case where $f_m = a_1\psi_1 + \dots + a_k \psi_k$, $k \in \mathbb{N}$, and $\mathcal{D}_0$ are used to select suitable values for the coefficients $\{a_i\}_{i=1}^k$ \citep[see e.g.][]{Rubinstein}.

To go beyond control variates and achieve an error of $o_{\mathrm{P}}(n^{-1/2})$, we must construct increasingly accurate approximations $f_m$ to $f$.
Indeed, under the scaling $m = O(n^\gamma)$, if the expected functional approximation error satisfies $\mathbb{E}_{\mathcal{D}_0}[\sigma^2(f - f_m)] = O(m^{-\delta})$ for some $\delta \geq 0$, then 
\begin{eqnarray}
\mathbb{E}_{\mathcal{D}_0} \mathbb{E}_{\mathcal{D}_1}\left[\left(I_{m,n} - \int f \d\Pi\right)^2\right] = O(n^{-1-\gamma\delta}). \label{analyse mse}
\end{eqnarray}
Here we have written $\mathbb{E}_{\mathcal{D}_0}$ for the expectation with respect to the sampling distribution $\Pi$ of the $m$ random variables that constitute $\mathcal{D}_0$.
The rate above is optimised by taking $\gamma = 1$, so that an optimal sample-split satisfies $m/n \rightarrow \rho$ for some $\rho \in (0,1]$ as $n \rightarrow \infty$; this will be assumed in the sequel.

When $\Pi$ is given via an un-normalised density, this framework can only be exploited if it is possible to construct approximations $f_m$ whose integrals $\int f_m \d\Pi$ are available in closed-form.
If and when this is possible, the term in parentheses in Eqn. \ref{splitting estimators} is known as a \emph{control functional}.
\cite{Oates} showed how to build a flexible class of control functionals based on Stein's method; the key points are presented next.

\subsection{Stein Operator}

To begin, we make the following assumptions on the density $\pi$:
\begin{enumerate}
\item[(A1)] $\pi \in C^{a+1}(\mathcal{X} \cup \partial \mathcal{X})$ for some $a \in \mathbb{N}_0$.
\item[(A2)] $\pi > 0$ in $\mathcal{X}$. 
\end{enumerate}
The gradient function $\nabla_{\bm{x}} \log \pi(\cdot)$ is well-defined and $C^a(\mathcal{X} \cup \partial \mathcal{X})$ by (A1,2).
Crucially, gradients can be evaluated even when $\pi$ is only available un-normalised.
Consider the following Stein operator:
\begin{eqnarray}
\mathbb{S}_\pi \; : \; C^1(\mathcal{X}) \times \dots \times C^1(\mathcal{X}) & \rightarrow & C^0(\mathcal{X}) \nonumber \\
\bm{\phi}(\cdot) & \mapsto & \mathbb{S}_\pi[\bm{\phi}](\cdot) :=  \nabla_{\bm{x}} \cdot \bm{\phi}(\cdot) + \bm{\phi}(\cdot) \cdot \nabla_{\bm{x}} \log \pi(\cdot) \label{Stein operator}
\end{eqnarray}
This definition can be motivated in several ways, including via Schr\"{o}dinger Hamiltonians \citep{Assaraf} and via the generator method of Barbour applied to an overdamped Langevin diffusion \citep{Gorham}.
The choice of Stein operator is not unique and some alternatives are listed in \cite{Gorham2}.

For functional approximation we follow \cite{Oates} and study approximations of the form
\begin{eqnarray}
f_m(\cdot) & := & \beta + \mathbb{S}_\pi[\bm{\phi}](\cdot) \label{CFs defin}
\end{eqnarray}
where $\beta \in \mathbb{R}$ is a constant and $\mathbb{S}_{\pi}[\bm{\phi}](\cdot)$ acts as a flexible function, parametrised by the choice of $\bm{\phi} \in C^1(\mathcal{X}) \times \dots \times C^1(\mathcal{X})$.
Under regularity assumptions introduced below, integration-by-parts (Eqn. \ref{Stein}) can be applied to obtain $\int \mathbb{S}_\pi[\bm{\phi}] \d\Pi = 0$ (Lemma \ref{lem integrate to zero}).
Thus, for this class of functions, $\int f_m \d \Pi$ permits a trivial closed-form and $\mathbb{S}_\pi[\bm{\phi}]$ is a control functional (i.e. integrates to 0).

The choice of $\beta$ and $\bm{\phi}$ can be cast as an optimisation problem over a Hilbert space and this will be the focus next.

\subsection{Stein Operators on Hilbert Spaces} \label{reproducing kernel Hilbert space construct}

This section formulates the construction of $f_m$ as approximation in a Hilbert space $\mathcal{H}_+ \subset L^2(\mathcal{X},\Pi)$.
This construction first appeared in \cite{Oates} and was subsequently explored in several papers \citep[e.g.][]{Liu,Chwialkowski,Gorham3}.

First we restrict each component function $\phi_i : \mathcal{X} \rightarrow \mathbb{R}$ to belong to a Hilbert space $\mathcal{H}$ with inner product $\langle \cdot,\cdot \rangle_{\mathcal{H}}$.
Moreover we insist that $\mathcal{H}$ is a (non-trivial) reproducing kernel Hilbert space (RKHS), i.e. there exists a (non-zero) symmetric positive definite function $k : \mathcal{X} \times \mathcal{X} \rightarrow \mathbb{R}$ such that (i) for all $\bm{x} \in \mathcal{X}$ we have $k(\cdot,\bm{x}) \in \mathcal{H}$ and (ii) for all $\bm{x} \in \mathcal{X}$ and $h \in \mathcal{H}$ we have $h(\bm{x}) = \langle h , k(\cdot,\bm{x}) \rangle_{\mathcal{H}}$ \citep[see][for background]{Berlinet}.
The vector-valued function $\bm{\phi} : \mathcal{X} \rightarrow \mathbb{R}^d$ is defined in the Cartesian product space $\mathcal{H}^d := \mathcal{H} \times \dots \times \mathcal{H}$, itself a Hilbert space with the inner product $\langle \bm{\phi}, \bm{\phi}' \rangle_{\mathcal{H}^d} = \sum_{i=1}^d \langle \phi_i, \phi_i' \rangle_{\mathcal{H}}$.

To ensure $\mathcal{H} \subseteq C^1(\mathcal{X})$ we make an assumption on $k$ that will be enforced by construction through selection of the kernel:
\begin{enumerate}
\item[(A3)] $k \in C_2^{b+1}(\mathcal{X} \cup \partial \mathcal{X})$ for some $b \in \mathbb{N}_0$.
\end{enumerate}

\subsubsection{Boundary Conditions}

Two further assumptions are made on $\pi$.
To this end, denote by $\mathcal{Q}(k)$ the set of densities $q = \d{} Q / \d{} \Lambda$ on $(\mathcal{X} \cup \partial \mathcal{X}, \mathcal{B})$ such that (a) $q \in C^1(\mathcal{X} \cup \partial \mathcal{X})$, (b) $q > 0$ in $\mathcal{X}$, and (c) for all $i = 1,\dots,d$ we have $\nabla_{x_i} \log q \in L^2(\mathcal{X} \cup \partial \mathcal{X},Q')$ for all distributions $Q'$ on $(\mathcal{X} \cup \partial \mathcal{X},\mathcal{B})$.
Let $\mathcal{R}(k)$ denote the set of densities $q$ for which $q(\bm{x}) k(\bm{x} , \cdot) = 0$ for all $\bm{x} \in \partial \mathcal{X}$.
\begin{enumerate}
\item[(A$\bar{2}$)] $\pi \in \mathcal{Q}(k)$
\item[(A4)] $\pi \in \mathcal{R}(k)$
\end{enumerate}
The assumption (A$\bar{2}$) was first discussed in \cite{Chwialkowski}; note in particular that (A$\bar{2}$) implies (A2).
A constructive approach to ensure (A4) holds is to start with an arbitrary RKHS $\tilde{\mathcal{H}}$ with reproducing kernel $\tilde{k}$ and let $B : \tilde{\mathcal{H}} \rightarrow \text{im}(B)$ be a linear operator such that $B \varphi(\bm{x}) := \delta(\bm{x}) \varphi(\bm{x})$, where $\delta(\cdot)$ is a smooth function such that $\pi(\cdot) \delta(\cdot)$ vanishes on $\partial \mathcal{X}$.
Then $\mathcal{H} = \text{im}(B)$ is a RKHS whose kernel $k$ is defined by $k(\bm{x},\bm{x}') = \delta(\bm{x}) \delta(\bm{x}') \tilde{k}(\bm{x},\bm{x}')$.
This construction will be used in Sec. \ref{illustration}.
The following Lemma shows that $\mathbb{S}_\pi[\bm{\phi}]$ is a control functional:

\begin{lemma} \label{lem integrate to zero}
Under (A1-4), if $\bm{\phi} \in \mathcal{H}^d$ then $\int \mathbb{S}_\pi[\bm{\phi}] \d \Pi = 0$.
\end{lemma}

Now, consider the set $\mathcal{H}_0 := \mathbb{S}_\pi [\mathcal{H}^d]$, whose elements $\mathbb{S}_\pi[\bm{\phi}]$ result from application of the Stein operator $\mathbb{S}_\pi$ to elements $\bm{\phi}$ of the Hilbert space $\mathcal{H}^d$.
\citet[][Thm. 1]{Oates} showed that $\mathcal{H}_0$ can be endowed with the gradient-based reproducing kernel
\begin{eqnarray}
k_0(\bm{x},\bm{x}') & := & (\nabla_{\bm{x}} \cdot \nabla_{\bm{x}'}) \; k(\bm{x},\bm{x}') + (\nabla_{\bm{x}} \log \pi(\bm{x})) \cdot (\nabla_{\bm{x}'} k(\bm{x},\bm{x}')) \label{k0 expression} \\ 
& & + \; (\nabla_{\bm{x}'} \log \pi(\bm{x}')) \cdot (\nabla_{\bm{x}} k(\bm{x},\bm{x}')) + (\nabla_{\bm{x}} \log \pi(\bm{x})) \cdot (\nabla_{\bm{x}'} \log \pi(\bm{x}')) \; k(\bm{x},\bm{x}'). \nonumber
\end{eqnarray}
From (A1,$\bar{2}$,3) it follows that $\mathcal{H}_0 \subseteq C^{a \wedge b}(\mathcal{X} \cup \partial \mathcal{X})$.
Moreover, under (A1,$\bar{2}$,3,4), the kernel $k_0$ satisfies
$\int k_0(\bm{x},\bm{x}') \Pi( \d {\bm{x}} ) = 0$ for all $\bm{x}' \in \mathcal{X}$.
Indeed, the function $k_0(\cdot,\bm{x}')$ belongs to $\mathcal{H}_0$ by definition and Lemma \ref{lem integrate to zero} shows that all elements of $\mathcal{H}_0$ have zero integral.

\subsubsection{Approximation in $\mathcal{H}_+$} \label{con approx}

Now we can be specific about how $\beta$ and $\bm{\phi}$ are selected.
Write $\mathcal{H}_{\mathbb{R}}$ for the RKHS of constant functions, characterised by the kernel $k_{\mathbb{R}}(\bm{x},\bm{x}') = c$, $c > 0$, for all $\bm{x},\bm{x}' \in \mathcal{X}$.
Denote the norms associated to $\mathcal{H}_{\mathbb{R}}$ and $\mathcal{H}_0$ respectively by $\|\cdot\|_{\mathcal{H}_{\mathbb{R}}}$ and $\|\cdot\|_{\mathcal{H}_0}$.
Write 
$$
\mathcal{H}_+ := \mathcal{H}_{\mathbb{R}} + \mathcal{H}_0 =\{\beta + \psi : \beta \in \mathcal{H}_{\mathbb{R}}, \; \psi \in \mathcal{H}_0\}.
$$
Equip $\mathcal{H}_+$ with the norm $\|f\|_{\mathcal{H}_+}^2 := \|\beta\|_{\mathcal{H}_{\mathbb{R}}}^2 + \|\psi\|_{\mathcal{H}_0}^2$.
It can be shown that $\mathcal{H}_+$ is a RKHS with kernel $k_+(\bm{x},\bm{x}') := k_{\mathbb{R}}(\bm{x},\bm{x}') + k_0(\bm{x},\bm{x}')$ \citep[][Thm. 5, p24]{Berlinet}.
From (A1-3) it follows that $\mathcal{H}_+ \subseteq C^{a \wedge b}(\mathcal{X})$.

The choice of $\beta$ and $\bm{\phi}$ is cast as a least-squares optimisation problem:
\begin{eqnarray*}
f_m \; := \; \arg\min \; \|h\|_{\mathcal{H}_+}^2 \; \text{ s.t. } \forall \; i = 1,\dots,m, \; h(\bm{x}_i) = f(\bm{x}_i), \quad h \in \mathcal{H}_+. 
\end{eqnarray*}
By the representer theorem \citep{Scholkopf2} we have $f_m(\bm{x}) = \sum_{i=1}^m a_i k_+(\bm{x},\bm{x}_i)$ where the coefficients $\mathbf{a} = [a_1,\dots,a_m]^\top$ are the solution of the linear system $\mathbf{K}_+ \mathbf{a} = \mathbf{f}_0$ where $\mathbf{K}_+ \in \mathbb{R}^{m \times m}$, $[\mathbf{K}_+]_{i,j} = k_+(\bm{x}_{i},\bm{x}_{j})$, $\mathbf{f}_0 \in \mathbb{R}^{m \times 1}$, $[\mathbf{f}_+]_i = f(\bm{x}_i)$.
In situations where $\mathbf{K}_+$ is not full-rank, we define $f_m \equiv 0$.
Numerical inversion of this system is associated with a $O(m^3)$ cost and may in practice require additional numerical regularisation; this is relatively standard.

\subsection{Theoretical Results} \label{consist approx asym}

Our novel analysis, next, builds on results from the scattered data approximation literature \citep{Wendland} and the study of the stability properties of Markov chains \citep{Meyn}.

\subsubsection{The Case of Independent Samples}

First we focus on scattered data approximation and state two assumptions that are central to our analysis:
\begin{enumerate}
\item[(A5)] $\pi > 0$ on $\mathcal{X} \cup \partial \mathcal{X}$
\item[(A6)] $f \in \mathcal{H}_+$. 
\end{enumerate}
Here (A5) extends (A2) in requiring also that $\pi>0$ on $\partial \mathcal{X}$.
(A6) ensures that the problem is well-posed.
Define the fill distance
$$
h_{\mathcal{D}_0} := \sup_{\bm{x} \in \mathcal{X}} \; \min_{i = 1,...,m} \|\bm{x} - \bm{x}_i\|_2.
$$
The proof strategy that we present here decomposes into two parts; (i) first, error bounds are obtained on the functional approximation error $\sigma^2(f - f_m)$ in terms of the fill distance $h_{\mathcal{D}_0}$, (ii) second, the fill distance $h_{\mathcal{D}_0}$ is shown to vanish under sampling (with high probability).
For (ii) to occur, we require an additional constraint on the geometry of $\mathcal{X}$:
\begin{enumerate}
\item[(A7)] The domain $\mathcal{X} \cup \partial \mathcal{X}$ satisfies an {\it interior cone condition}, i.e. there exists an angle $\theta \in (0,\pi/ 2)$ and a radius $r > 0$ such that for every $\bm{x} \in \mathcal{X} \cup \partial \mathcal{X}$ there exists a unit vector $\bm{\xi}$ such that the cone
$$
\mathcal{C}(\bm{x},\bm{\xi},\theta,r) := \{\bm{x} + \lambda\bm{y} \; : \; \bm{y} \in \mathbb{R}^d, \; \|\bm{y}\|_2 = 1, \; \bm{y}^\top\bm{\xi} \geq \cos\theta , \; \lambda \in [0,r]\}
$$
is contained in $\mathcal{X} \cup \partial \mathcal{X}$.
\end{enumerate}
The purpose of (A7) is to rule out the possibility of `pinch points' on $\partial\mathcal{X}$ (i.e. $\prec$-shaped regions), since intuitively sampling-based approaches can fail to `get into the corners' of the domain.
The limiting behaviour of the fill distance under sampling enters through the following technical result:
\begin{lemma} \label{tech lemma}
Let $g : [0,\infty) \rightarrow [0,\infty)$ be continuous, monotone increasing, and satisfy $g(0) = 0$ and $\lim_{x \downarrow 0} g(x) \exp(x^{-3d}) = \infty$.
Then under (A5,7) we have 
$$
\mathbb{E}_{\mathcal{D}_0} [g(h_{\mathcal{D}_0})] = O( g(m^{- \frac{1}{d} + \epsilon}) ), 
$$
where $\epsilon > 0$ can be arbitrarily small.
\end{lemma}

Our first main result can now be stated:
\begin{theorem} \label{independent} 
Assume (A1,$\bar{2}$,3-7).
Recall that we partition the set $\mathcal{D}$ as $\mathcal{D}_0 \cup \mathcal{D}_1$ where $|\mathcal{D}_0| = m$ and $|\mathcal{D}_1| = n-m$.
There exists $h > 0$, independent of $m,n$, such that the estimator $I_{m,n}$ is an unbiased estimator of $\int f \d\Pi$ with 
$$
\mathbb{E}_{\mathcal{D}_0} \mathbb{E}_{\mathcal{D}_1}\left[ 1_{h_{\mathcal{D}_0} < h} \left(I_{m,n} - \int f \d\Pi\right)^2\right] = O\left( (n-m)^{-1} m^{-2 \frac{a \wedge b}{d} + \epsilon} \right) 
$$
where $\epsilon > 0$ can be arbitrarily small.
\end{theorem}
\noindent Thus for $m = O(n)$, this result establishes an overall error of $O(n^{-1 - 2\frac{a \wedge b}{d} + \epsilon})$, as claimed.
This establishes that these estimates are more efficient than standard Monte Carlo estimators when $a \wedge b > 0$.
Or, when the cost of solving a linear system is taken into account, the method is more efficient on a per-cost basis when $a \wedge b > d$.
This provides new insight into the first set of empirical results reported in \cite{Oates} where, for assessment purposes, samples were generated independently from known, smooth densities. There, control functionals were constructed based on smooth kernels and integration errors were shown to be substantially reduced.

On the negative side, this result illustrates a curse of dimension that appears to be intrinsic to the method.
We return to this point in Sec. \ref{discuss}.

The results above hold for independent samples, yet the main area of application for control functionals is estimation based on the MCMC output.
In the next section we prove that the assumption of independence can be relaxed.

\subsubsection{The Case of Non-Independent Samples} \label{MCMC extension}

In practice, samples from posterior distributions are often obtained via MCMC methods.
Our analysis must therefore be extended to the non-independent setting:
Consider the case where $\{\bm{x}_i\}_{i=1}^n$ are generated by a reversible Markov chain targeting $\Pi$.
We make the following stochastic stability assumption:
\begin{enumerate}
\item[(A8)] The Markov chain is uniformly ergodic.
\end{enumerate}
\noindent Then our first step is to extend Lemma \ref{tech lemma} to the non-independent setting:
\begin{lemma} \label{tech lemma 2}
The conclusion of Lemma \ref{tech lemma} holds when $\{\bm{x}_i\}_{i=1}^n$ are generated via MCMC, subject to (A8).
\end{lemma}

Non-independence presents us with the possibility that two of the states $\bm{x}_i,\bm{x}_j \in \mathcal{D}_0$ are identical (for instance, when a Metropolis-Hastings sample is used and a rejection occurs).
Under our current definition, such an event would cause the kernel matrix $\mathbf{K}_+$ to become singular and the control functional to become trivial $f_m = 0$.
It is thus necessary to modify the construction.
Specifically, we assume that $\mathcal{D}_0$ has been pre-filtered such that any repeated states have been removed.
Note that this does not `introduce bias', since we are only pre-filtering $\mathcal{D}_0$, not $\mathcal{D}_1$.
This reduces the effective number $m$ of points in $\mathcal{D}_0$ by at most a constant factor and has no impact on the asymptotics.

With this technical point safely surmounted, we present our second main result:
\begin{theorem} \label{dependent}
The conclusion of Theorem \ref{independent} holds when $\{\bm{x}_i\}_{i=1}^n$ are generated via MCMC, subject to (A8).
\end{theorem}
\noindent This result again demonstrates that control functionals are more cost-efficient than standard Monte Carlo when $a \wedge b > d$ and that efficiency is limited by the rougher of the density $\pi$ and the test function $f$.
This helps to explain the second set of empirical results obtained in \cite{Oates}, where excellent performance was reported on problems that involved smooth densities, smooth kernels and MCMC sampling methods.
On the other hand, we again observe a curse of dimension that is inherent to control functionals and, indeed, control variates in general.

\subsection{Commentary} \label{remarks}

Several points of discussion are covered below, on the appropriateness of the assumptions, the strength of the results and aspects of implementation.

\vspace{5pt}
\noindent {\bf On the Assumptions:}
Assumptions (A1,$\bar{2}$,3,7) are not unduly restrictive.
The boundary condition (A4) has previously been discussed in \cite{Oates}.
Below we discuss the remaining assumptions, (A5,6,8).

Our entire analysis was predicated on (A5), the assumption that $\pi$ is bounded away from 0 on the compact set $\mathcal{X} \cup \partial\mathcal{X}$.
This ensured that $\pi$ was equivalent to Lebesgue measure on $\mathcal{X} \cup \partial \mathcal{X}$ and enabled this change of measure in the proofs.
This is clearly a restrictive set-up as certain distributions of interest do vanish, however the assumption was intrinsic to our theoretical approach.

Our analysis also relied on (A6), i.e. that $f$ belongs to the function space $\mathcal{H}_+$.
It it is thus natural to examine this assumption in more detail.
To this end, we provide the following lemma.
Recall that a RKHS $\mathcal{H}$ is \emph{$c$-universal} if it is dense as a set in $(C^0(\mathcal{X} \cup \partial \mathcal{X}),\|\cdot\|_\infty)$.
\begin{lemma} \label{characteristic}
Assume (A$\bar{2}$,3,4).
If $\mathcal{H}$ is c-universal then $\mathcal{H}_+$ is dense as a set in $(L^2(\mathcal{X} \cup \partial \mathcal{X},\Pi),\|\cdot\|_2)$.
\end{lemma}
\noindent The notion of $c$-universality was introduced by \cite{Steinwart3}, who showed that many widely-used kernels are $c$-universal on compact sets.
Indeed, Prop. 1 of \cite{Micchelli} proves that a RKHS with kernel $k$ is c-universal if and only if the map $\Pi' \mapsto \Pi'[k(\cdot,\cdot)]$, from the space of finite signed Borel measures $\Pi'$ to the RKHS $\mathcal{H}$, is injective, which is a weak requirement.
It is \emph{not}, however, clear whether (A4), (A5) can both hold when $k$ is also $c$-universal.
Further work will therefore be required to better assess the consequences of $f \notin \mathcal{H}_+$.
This might proceed in a similar vein to the related work of \cite{Narcowich,Kanagawa}.

The last assumption to discuss is (A8); uniform ergodicity of the Markov chain.
Since $\pi$ is absolutely continuous with respect to Lebesgue measure on $\mathcal{X} \cup \partial \mathcal{X}$, in practice any Markov chain that targets $\Pi$ will typically be uniformly ergodic.
Indeed, \cite{Roberts3} constructed an example where a pinch point in the domain caused a Gibbs sampler targeting a uniform distribution to fail to be geometrically ergodic; their construction violates our (A7).

\vspace{5pt}
\noindent {\bf On the Results:}
The intuition for the results in Thms. \ref{independent} and \ref{dependent} can be described as `accurate estimation with high probability', since the condition $h_{\mathcal{D}_0} < h$ is satisfied when the samples $\mathcal{D}_0$ cover the state space $\mathcal{X}$, which occurs with unit probability in the $m \rightarrow \infty$ limit.
There are two equivalent statements that can be made unconditionally on $h_{\mathcal{D}_0} < h$:
(i) Firstly, one can simply re-define $f_m = 0$ whenever $h \geq h_0$, i.e. when the states $\mathcal{D}_0$ are poorly spaced we revert to the usual Monte Carlo estimator.
(ii) Secondly, one could augment $\mathcal{D}_0$ with additional fixed states, such as a grid, $\{\bm{g}_i\}_{i=1}^G$, to ensure that $h_{\mathcal{D}_0} < h$ is automatically satisfied.
However, we find both of these equivalent approaches to be less aesthetically pleasing, since in practice this requires that $h$ be explicitly computed.

The condition $h_{\mathcal{D}_0} < h$ suggests that the asymptotics hold in the same regime where QMC methods could also be successful.
However, as explained in Sec. \ref{intro}, the method of \cite{Oates} carries some advantages over the QMC approach that could be important.
First, it provides unbiased estimation of $\int f \d\Pi$, which enables straight-forward empirical assessment.
Second, the fact that it is based on MCMC output renders it more convenient to implement.

On the sharpness of our results, we refer to Sec. 11.7 of \cite{Wendland} where an overview of the strengths and weaknesses of results in the scattered data approximation literature is provided.

\vspace{5pt}
\noindent {\bf On the Data-Split:}
It is required to partition samples into sets $\mathcal{D}_0$ and $\mathcal{D}_1$, whose sizes must be specified.
Substituting $\rho = m/n$ into the conclusion of Thm. \ref{independent} and minimising this expression over $\rho \in (0,1]$ leads to an optimal value 
\begin{eqnarray}
\rho^* & = & \frac{\nu}{1 + \nu} \hspace{20pt} \text{where} \hspace{20pt} \nu = 2 \frac{a \wedge b}{d}.
\end{eqnarray}
Thus, when $a \wedge b \gg d$ we have $\rho^* \approx 1$ and the optimal method is essentially a numerical quadrature method (i.e. all samples assigned to $\mathcal{D}_0$).
Conversely, when $a \wedge b \ll d$ we have $\rho^* \approx 0$ and the optimal method becomes a Monte Carlo method (i.e. all samples assigned to $\mathcal{D}_1$).

\vspace{5pt}
\noindent {\bf On the Bandwidth:}
For the experiments reported next we considered radial kernels of the form 
$$
\tilde{k}(\bm{x},\bm{x}') \; = \; \varphi\left(\frac{\|\bm{x} - \bm{x}'\|_2}{h}\right)
$$ 
where $h > 0$ is a bandwidth parameter and $\varphi$ is a radial basis function, to be specified.
An appropriate value for the bandwidth $h$ must therefore be selected.
An important consideration is that if $h$ is selected based on $\mathcal{D}_0$ but not on $\mathcal{D}_1$ then the estimator $I_{m,n}$ remains unbiased.
To this end, we propose to select $h$ via maximisation of the log-marginal likelihood
\begin{eqnarray*}
\log p(\mathbf{f}_0 | \mathcal{D}_0 , h) & = & - \frac{1}{2} \mathbf{f}_0^\top \mathbf{K}_+^{-1} \mathbf{f}_0 - \frac{1}{2} \log |\mathbf{K}_+| - \frac{m}{2} \log 2 \pi
\end{eqnarray*}
which arises from the duality with Gaussian processes and approximation in RKHS \citep[see e.g.][]{Berlinet}.

\vspace{5pt}
\noindent {\bf On an Extension:}
An extension of the estimation method was also considered. 
Namely, for each $i$ one can build an approximation $f^{(-i)} \in \mathcal{H}_+$ to be used as a control functional for $f(\bm{x}_i)$, based on $\mathcal{D} \setminus \{\bm{x}_i\}$.
This results in a leave-one-out (LOO) estimator
\begin{eqnarray}
I_n := \frac{1}{n} \sum_{i=1}^n f(\bm{x}_i) - \left( f^{(-i)}(\bm{x}_i) - \int f^{(-i)} \d \Pi \right)
\end{eqnarray}
that again remains unbiased.
The performance of $I_n$ can be expected to compare favourably with that of $I_{m,n}$, but the computational cost of $I_n$ is larger at $O(n^4)$.

\vspace{5pt}
\noindent {\bf On Computation:}
It is important to emphasise the ease with which these estimators can be implemented.
In the $c \rightarrow \infty$ limit, explicit evaluation of Eqn. \ref{splitting estimators} is particularly straight-forward:
\begin{eqnarray}
I_{m,n} = \frac{1}{n-m} \bm{1}^\top \left\{ \mathbf{f}_1 - \mathbf{K}_{10} \mathbf{K}_0^{-1} \left[ \mathbf{f}_0 - \left( \frac{\bm{1}^\top \mathbf{K}_0^{-1} \mathbf{f}_0 }{\bm{1}^\top \mathbf{K}_0^{-1} \bm{1} } \right) \bm{1} \right] \right\}  \label{matrix form}
\end{eqnarray}
where $\mathbf{f}_1 \in \mathbb{R}^{n-m \times 1}$, $[\mathbf{f}_1]_i = f(\bm{x}_{m+i})$, $\mathbf{K}_0 \in \mathbb{R}^{m \times m}$, $[\mathbf{K}_0]_{i,j} = k_0(\bm{x}_i,\bm{x}_j)$, $\mathbf{K}_{10} \in \mathbb{R}^{n-m \times m}$ and $[\mathbf{K}_{10}]_{i,j} = k_0(\bm{x}_{m+i},\bm{x}_{j})$.
An implementation called \verb+control_func.m+ is available on the Matlab File Exchange to download.

\section{Numerical Results} \label{illustration}

First, in Sec. \ref{sec: converge check}, we assessed whether the theoretical results are borne out in simulation experiments.
Then, in Sec. \ref{sec: application}, we applied the method to a topical parameter estimation problem in uncertainty quantification for a groundwater flow model.

\subsection{Simulation} \label{sec: converge check}

To construct a test-bed for the theoretical results we considered the simple case where $\Pi$ is the uniform distribution on $\mathcal{X} = [0,1]^d$.
The test functions that we considered took the form $f(\bm{x}) = 1 + \sin(2 \pi \omega x_1)$ where $\omega$ was varied to create a problem that was either `easy' ($\omega = 1$) or `hard' ($\omega = 3$).
The importance of the first coordinate $x_1$ aimed to reflect the `low effective dimension' phenomena that is often encountered.
From symmetry of the integrand, the true integral is 1.

For estimation we took the radial basis function $\varphi$ to have variable smoothness and compact support, as studied in \cite{Wendland2}.
Explicit formulae for the $\varphi$ and their derivatives are contained in the electronic supplement.
To enforce (A4) we took $\delta(\bm{x}) = \prod_{i=1}^d x_i (1 - x_i)$ which vanishes on $\partial \mathcal{X}$.
The data-split fraction $\rho$ and the bandwidth $h$ were each optimised as described in Sec. \ref{remarks}.
Optimisation for $h$ was performed through 10 iterations of the Matlab function \verb+fminbnd+ constrained to $h \in [0,10]$.

Three estimators were considered; the standard Monte Carlo estimator, the control functional (CF) estimator $I_{m,n}$ in Eqn. \ref{matrix form} and the LOO estimator $I_n$. In the case of the LOO estimator, the bandwidth $h$ was re-optimised in building each of the $n$ control functionals $f^{(-i)}$.
(A1,$\bar{2}$,3-5,7) were satisfied in this experiment.
Thus, for $f \in \mathcal{H}_+$, Thm. \ref{independent} entails a mean squared integration error for $I_{m,n}$ of $O(n^{-1-2\frac{b}{d} + \epsilon})$, since $\pi(\bm{x}) = 1 \in C^{a+1}$ for all $a \in \mathbb{N}_0$.
However, the theoretical analysis does not take into account automatic selection of the bandwidth $h$; this will be assessed through experiment.

\begin{figure}[t!]
\centering
\includegraphics[width = \textwidth,clip,trim = 0.5cm 8.5cm 1cm 7.5cm]{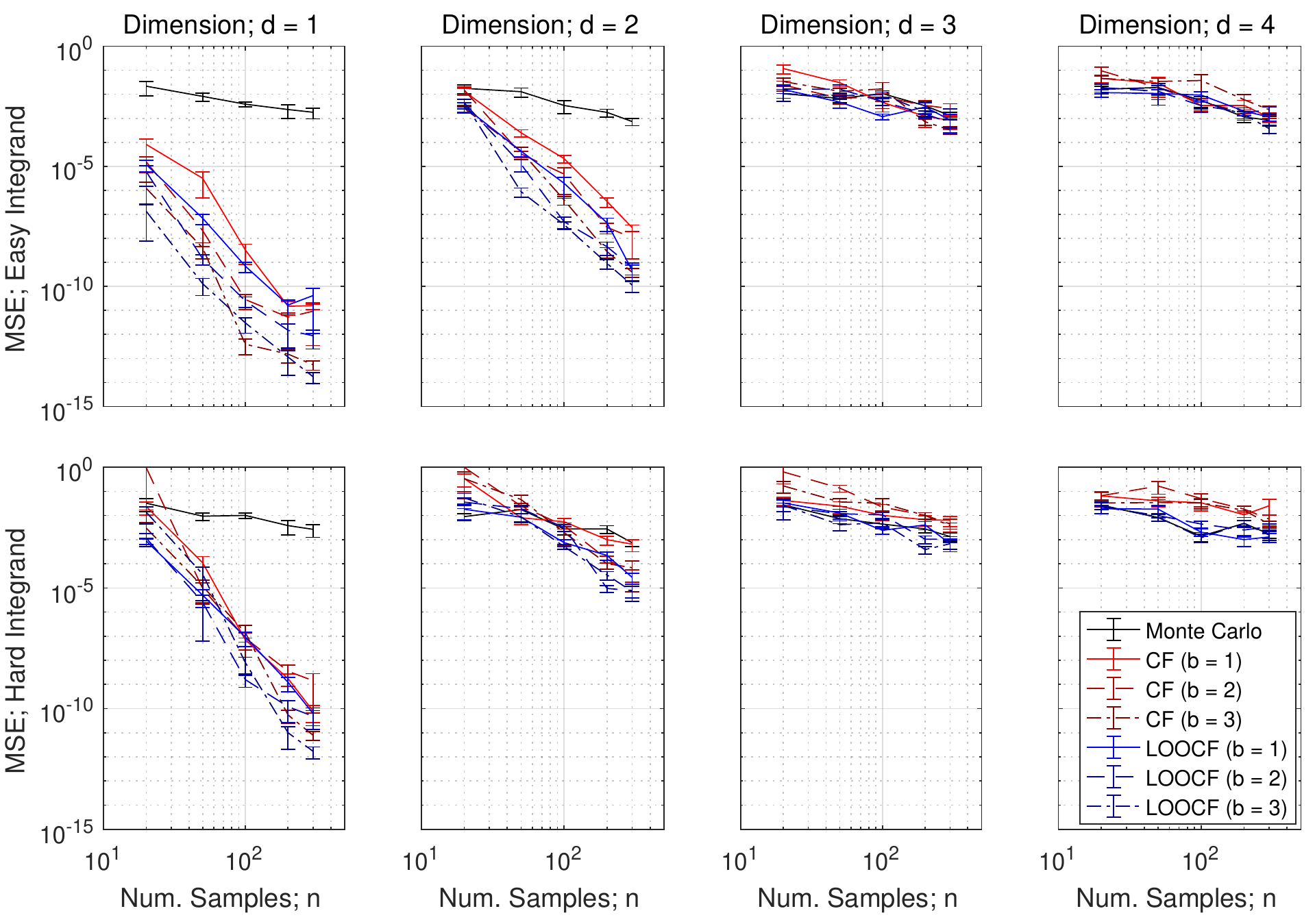}
\caption{
Simulation results; the case of independent samples.
An `easy' and a `hard' integrand were considered.
The mean square error (MSE) was estimated for the standard Monte Carlo estimator, the control functional (CF) estimator $I_{m,n}$ and the leave-one-out (LOOCF) estimator $I_n$, and plotted against the number $n$ of samples used.
The CF and LOOCF estimators were based on kernels of smoothness $b \in \{1,2,3\}$.
Standard errors are also displayed.}
\label{sim results}
\end{figure}

\vspace{5pt}
\noindent {\bf Independent Samples:}
To study estimator performance, we repeatedly generated collections of $n$ independent uniform random variables $\{\bm{x}_i\}_{i=1}^n$ and evaluated all three estimators on this set.
The procedure was repeated several times to obtain estimates (along with standard errors) for the average mean square errors (MSE) that were incurred.
Results are displayed in Fig. \ref{sim results}.
In these experiments the MSE appeared to decrease at least as rapidly as the rates that were predicted.
Also, as predicted, the estimator performance quickly deteriorated as the dimension $d$ was increased.
Indeed, for $d = 3,4$ an improvement over standard Monte Carlo was no longer observed.
The LOO estimator $I_n$ in general out-performed the CF estimator $I_{m,n}$, as expected, but at an increased computational cost.
The integration error was in general larger for the hard integrand.

\vspace{5pt}
\noindent {\bf Dependent Samples:}
The effect of correlation among the $\bm{x}_i$ was also explored.
For this, we considered a random walk $\bm{x}_i = \bm{x}_{i-1} + \bm{e}_i$ on the $d$-torus with $\{\bm{e}_i\}_{i=1}^n$ drawn uniformly on $[-\epsilon,\epsilon]^d$ and $\bm{x}_0 = \bm{0}$.
This is a Markov chain with invariant distribution $\Pi$.
The objective was to assess estimator performance as a function of the step size parameter $\epsilon$; results for $n = 100$ are shown in Fig. \ref{sim results2}.
Compared to Fig. \ref{sim results}, the MSE was larger in general when $\epsilon < 0.5$.
This reflects reduction in effective sample size of the set $\mathcal{D}_0$ used to build the control functional.

\begin{figure}[t!]
\centering
\includegraphics[width = \textwidth,clip,trim = 0.5cm 8.5cm 1cm 7.5cm]{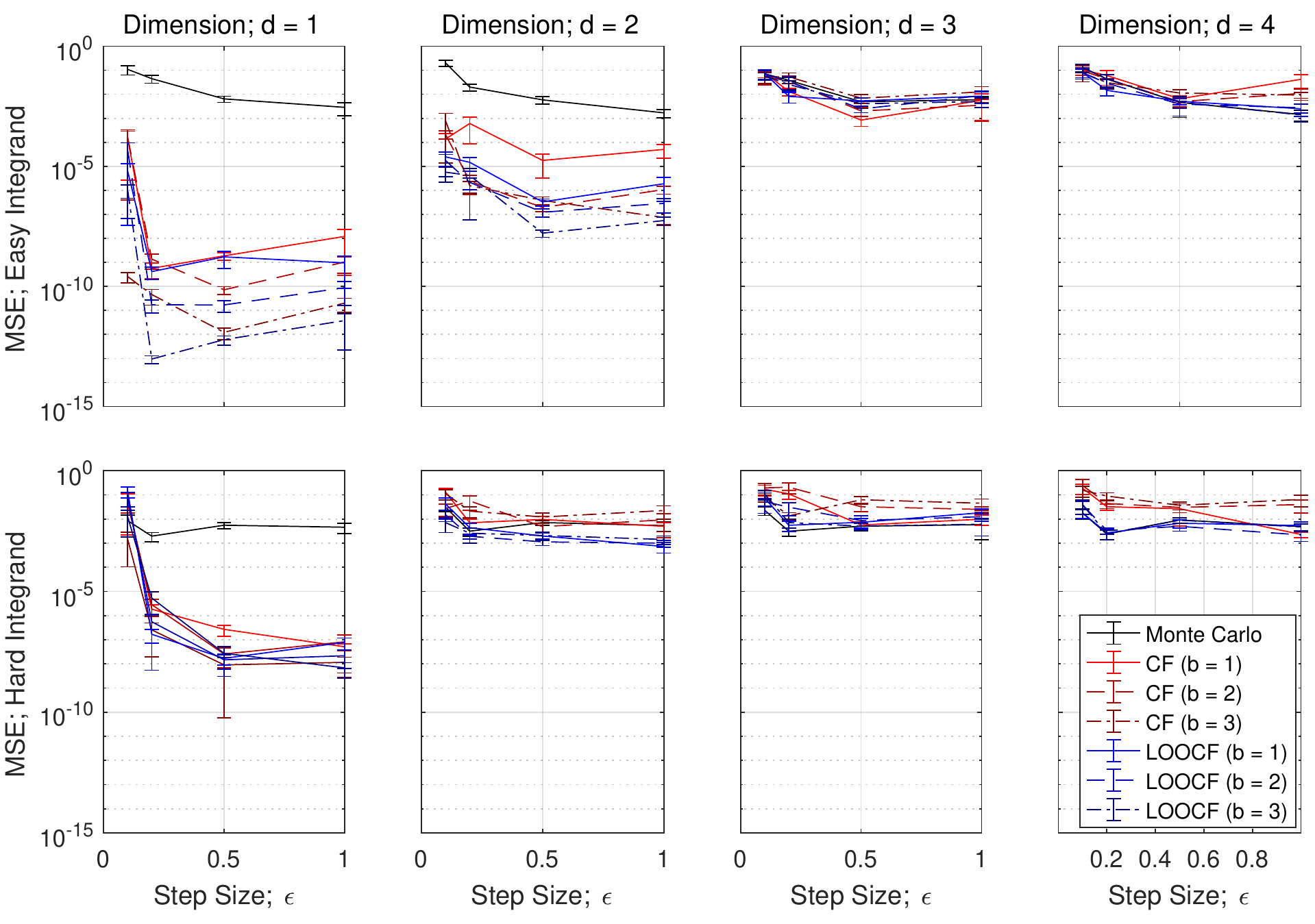}
\caption{
Simulation results; the case of dependent samples.
An `easy' and a `hard' integrand were considered.
The mean square error (MSE) was estimated for the standard Monte Carlo estimator, the control functional (CF) estimator $I_{m,n}$ and the leave-one-out (LOOCF) estimator $I_n$, where samples from a random walk of length $n = 100$ was used.
The MSE was plotted against the step size $\epsilon$ of the random walk.
The CF and LOOCF estimators were based on kernels of smoothness $b \in \{1,2,3\}$.
Standard errors are also displayed.}
\label{sim results2}
\end{figure}

\subsection{Application to Partial Differential Equations} \label{sec: application}

Our theoretical results are illustrated with a novel application to an inverse problem arising in a partial differential equation (PDE) model.
Specifically, we considered the following elliptic diffusion problem with mixed Dirichlet and Neumann boundary conditions:
\begin{eqnarray*}
\nabla_{\bm{x}} \cdot [\kappa(\bm{x};\bm{\theta}) \nabla_{\bm{x}} w(\bm{x})] & = & 0 \hspace{50pt} \text{if } x_1,x_2 \in (0,1) \\
w(\bm{x}) & = & \left\{ \begin{array}{ll} x_1 & \hspace{5pt} \text{if } x_2 = 0 \\ 1 - x_1 & \hspace{5pt} \text{if } x_2 = 1 \end{array} \right. \\
\nabla_{x_1} w(\bm{x}) & = & 0 \hspace{46.5pt} \text{ if } x_1 \in \{0,1\}.
\end{eqnarray*}
This PDE serves as a simple model of steady-state flow in aquifers and other subsurface systems; $\kappa$ can represent the permeability of a porous medium while $w$ represents the hydraulic head.
The aim is to make inferences on the field $\kappa$ in a setting where the underlying solution $w$ is observed with noise on a regular grid of $M^2$ points $\bm{x}_{i,j}$, $i,j = 1,\dots,M$.
The observation model $p(\bm{y}|\bm{\theta})$ takes the form $\bm{y} = \{y_{i,j}\}$ where $y_{i,j} = w(\bm{x}_{i,j}) + \epsilon_{i,j}$ and $\epsilon_{i,j}$ are independent normal random variables with standard deviation $\sigma = 0.1$.

\begin{figure}[t!]
\centering
\includegraphics[width = \textwidth]{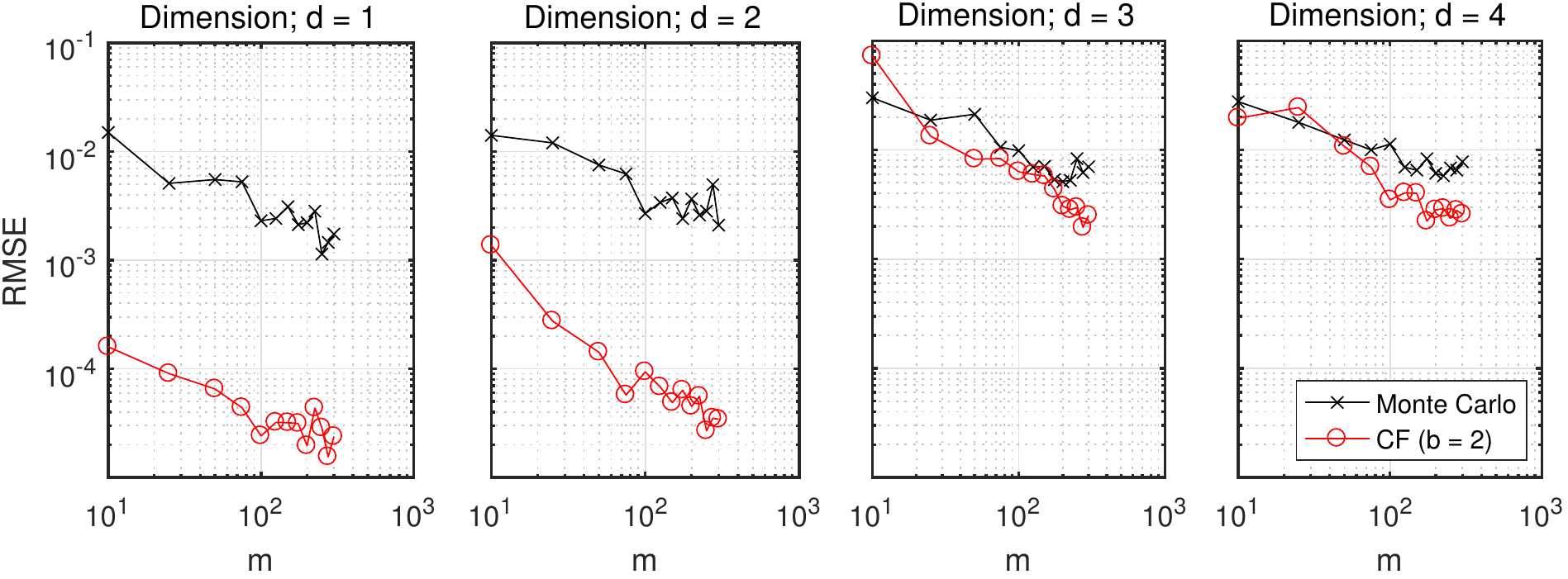}
\caption{Experimental results; an experiment to approximate the posterior mean of the parameters $\bm{\theta} \in [-10,10]^d$ that govern a permeability field. The figure shows root mean square error (RMSE) for (i) the standard Monte Carlo estimator based on $2m$ posterior samples, and (ii) the control functional (CF) estimator, where $m$ samples are used to train the control functional and the remaining $m$ samples are used to estimate the expectation. [Results are shown for the first parameter $\theta_1$; results for other parameters were similar. 
The Mat\'{e}rn kernel of order $7/2$ was employed; $b = 2$ in our notation.]}
\label{Matern}
\end{figure}

Following \cite{Stuart2}, the field $\kappa$ was endowed with a prior distribution of the form $\log \kappa(\bm{x};\bm{\theta}) = \sum_{i = 1}^d \theta_i \kappa_i(\bm{x})$, where the $\kappa_i$ are Fourier basis functions and $\theta_i$ are their associated coefficients.
For the inference we imposed a uniform prior $p(\bm{\theta}) \propto 1$ over the domain $[-10,10]^d$.
Our aim was to obtain accurate estimates for the posterior mean of the parameter $\bm{\theta}$.
The posterior density $p(\bm{\theta}| \bm{y}) \propto p(\bm{\theta}) p(\bm{y}|\bm{\theta})$ is available up to an unknown normalising constant $p(\bm{y})$.
Each evaluation of the likelihood necessitates the solution of the PDE; control functionals offer the possibility to reduce the number of likelihood evaluations, and hence the computational cost, required to achieve a given estimator precision.

As an aside, we note that the standard approach to inference employs a numerical integrator for the forward-solve, typically based on finite element methods.
This would provide us with gradient information on the posterior, but would also introduce some bias due to discretisation error.
To ensure that we obtain exact gradient information, we instead exploited a probabilistic meshless method due to \cite{Cockayne} as our numerical integrator.
Automatic differentiation was performed using the \verb+Autograd+ package \citep{Maclaurin}.

The key assumptions of our theory were verified.
Smoothness of the prior, together with ellipticity, imply (A1) holds for all $a \in \mathbb{N}$.
(A$\bar{2}$,5) hold since the prior and likelihood are well-behaved.
(A7) holds since the domain of integration was a hyper-cuboid.
Samples from the posterior $p(\bm{\theta}|\bm{y})$ were obtained using a Metropolis-adjusted Langevin sampler with fixed proposal covariance; this ensured that (A8) was satisfied.
Remaining assumptions were satisfied by construction of the kernel $k$:
Following the approach outlined in Section \ref{reproducing kernel Hilbert space construct}, we took $\tilde{k}(\bm{\theta},\bm{\theta}')$ to be the standard Mat\'{e}rn kernel of order $\frac{7}{2}$, so that $b =2$, and then formed $k(\bm{\theta},\bm{\theta}')$ as the product of $\tilde{k}(\bm{\theta},\bm{\theta}')$ and $\delta(\bm{\theta}) \delta(\bm{\theta}')$, where the boundary function $\delta$ satisfies $\delta(\bm{\theta}) = 1$ on $\bm{\theta} \in [-9,9]^d$, $\delta(\bm{\theta}) = 0$ when $\theta_i \in \{-10,10\}$ for some $i$, and $\delta$ was infinitely differentiable on $[-10,10]^d$.
With this construction, (A3) holds.
(A4) holds since $k$ has a root at $\theta_i \in \{-10,10\}$ for each $i$. 
The constant $c=1$ was fixed.
However the conclusion of Lemma \ref{characteristic} cannot be directly applied here since $\mathcal{H}$ is not $c$-universal ($k$ vanishes at $\theta_i = \pm 10$).

Observations were generated from the model with data-generating parameter $\bm{\theta} = \bm{1}$ and collected over a coarse grid of $M^2 = 36$ locations.
Samples of size $n$ were obtained from the posterior and divided equally between the training set $\mathcal{D}_0$ and test set $\mathcal{D}_1$.
The performance of gradient-based control functionals was benchmarked against that of standard Monte Carlo with all $n$ samples used.
We note that, in all experiments, all values of $\bm{\theta}$ encountered were contained in $[-9,9]^d$.
Thus it does not matter that we did not specify $\delta$ explicitly above, emphasising the weakness of assumption (A4) in practical application.

Results are shown in Figure \ref{Matern}.
For dimensions $d = 1$ and $2$, the estimator that uses control functionals achieved a dramatic reduction in asymptotic variance compared to the Monte Carlo benchmark.
On the other hand, for $d = 3,4$, the curse of dimension is clearly evident for the control functional method.

\section{Conclusion} \label{discuss}

This paper has established novel asymptotic analysis for a class of estimators based on Stein's method.
Our analysis makes explicit the contribution of the smoothness $a$ of the distribution $\Pi$, the smoothness $b$ of the test function $f$ and the dimension $d$ of the domain of integration.
As such, these results provide a rigorous theoretical explanation for the excellent performance in low-dimensions observed in previous work.

Several extensions of this work are suggested:
(i) Our results focused on compact domains, since this is the usual setting for results in the scattered data approximation literature.
However, the estimation method does not itself require that the domain of integration be compact.
Extending this analysis to the unbounded-domain setting appears challenging at present and remains a goal for future research.
(ii) Alternative literatures to the scattered data literature could form the basis of an analysis of control functionals, such as e.g. recent work by \cite{Migliorati}.
These efforts have the advantage of providing $L^2$ error bounds, rather than $L^\infty$ error bounds and might facilitate the extension to unbounded domains.
(iii) Generally, our theoretical results clarify the need to develop estimation strategies that do not suffer from the curse of dimension.
While this curse is intrinsic to functional approximation in general, due to the need to explore the state space, the observation that many test functions of interest are of low `effective dimension' suggests that more regularity on the function space could reasonably be assumed.
(iv) Recent work in \cite{Liu4} imposed an additional constraint on the coefficients $a_i$ in Sec. \ref{con approx}.
It would be interesting to extend our analysis to this context.

\vspace{20pt}

\noindent {\bf Acknowledgements:} The authors wish to thank Aretha Teckentrup, Motonobu Kanagawa, Lester Mackey and anonymous referees for their useful feedback.
CJO was supported by the ARC Centre of Excellence for Mathematical and Statistical Frontiers.
CJO and MG were supported by the Lloyds-Turing Programme on Data-Centric Engineering.
FXB was supported by EPSRC [EP/L016710/1]. 
MG was supported by the EPSRC grants [EP/J016934/3, EP/K034154/1, EP/P020720/1], an EPSRC Established Career Fellowship, the EU grant [EU/259348] and a Royal Society Wolfson Research Merit Award.

\newpage

\section{Supplement}

In this electronic supplement we establish correctness of the theoretical results in the main text.

\subsection{Additional Notation}

Let $\|f\|_2 := (\int f^2 \d\Pi)^{1/2}$.
Write $\text{vol}(A) := \int 1_A \d\Lambda$ and $\d{\bm{x}} := \d \Lambda(\bm{x})$.
Let $L^\infty(\mathcal{X})$ denote the set of measurable functions for which $\sup_{\bm{x} \in \mathcal{X}} |f(\bm{x})| < \infty$.
The notation $\oint_{\partial\mathcal{X}}$ will be used to denote a surface integral over $\partial\mathcal{X}$, $\bm{n}(\bm{x})$ to denote the unit normal vector to $\partial\mathcal{X}$ and $S(\d{} \bm{x})$ to denote a surface element.

\subsection{Proofs}

\begin{proof}[Proof of Lemma \ref{lem integrate to zero}]
From the Moore-Aronszajn theorem \citep{Aronszajn}, $\mathcal{H}^d$ can be expressed as
$$
\mathcal{H}^d = \left\{ \bm{\phi} : \mathcal{X} \rightarrow \mathbb{R}^d \text{ s.t. } \phi_i(\bm{x}) = \sum_{j=1}^\infty a_{i,j} k(\bm{x},\bm{x}_{i,j}) \text{ where } \sum_{j=1}^\infty a_{i,j}^2 k(\bm{x}_{i,j},\bm{x}_{i,j}) < \infty \right\}.
$$
Note that from (A3) the kernel $k$ is continuous and thus bounded on the compact set $\mathcal{X} \cup \partial \mathcal{X}$.
It follows that the series representation given in the Moore-Aronszajn theorem is uniformly convergent.
Then, for $\bm{\phi} \in \mathcal{H}^d$ represented in Moore-Aronszajn form,
\begin{eqnarray*}
\int \mathbb{S}_\pi[\bm{\phi}](\bm{x}) \Pi(\d{} \bm{x}) & = & \int \sum_{i=1}^d (\nabla_{x_i} + \nabla_{x_i} \log \pi(\bm{x})) \sum_{j=1}^\infty a_{i,j} k(\bm{x},\bm{x}_{i,j}) \; \Pi(\d{} \bm{x}) \\
& = & \sum_{i=1}^d \sum_{j=1}^\infty a_{i,j} \int (\nabla_{x_i} + \nabla_{x_i} \log \pi(\bm{x})) k(\bm{x},\bm{x}_{i,j}) \; \Pi(\d{} \bm{x}) \\
& = & \sum_{i=1}^d \sum_{j=1}^\infty a_{i,j} \int_{\mathcal{X} \cup \partial \mathcal{X}} \nabla_{x_i} \{ \pi(\bm{x}) k(\bm{x},\bm{x}_{i,j}) \} \; \mathrm{d} \bm{x} \\
& = & \sum_{i=1}^d \sum_{j=1}^\infty a_{i,j} \oint_{\partial\mathcal{X}} \underbrace{ \pi(\bm{x}) k(\bm{x},\bm{x}_{i,j}) }_{(*)} n_i(\bm{x}) S(\mathrm{d} \bm{x}) \quad = \quad 0,
\end{eqnarray*}
where the order of summation and calculus operations can be interchanged due to uniform convergence of the series representation on the compact set $\mathcal{X} \cup \partial \mathcal{X}$.
The term $(*)$ equals zero for all $\bm{x}_{i,j} \in \mathcal{X} \cup \partial \mathcal{X}$ by (A4).
\end{proof}

\begin{proof}[Proof of Lemma \ref{tech lemma}]
In rigorously establishing this result there are six main steps.
Initially we fix $\bm{x} \in \mathcal{X} \cup \partial \mathcal{X}$ and aim to show that with high probability there exists a state $\bm{x}_j$ close to $\bm{x}$.
Then we consider implications for the distribution of the fill distance.

\vspace{5pt}
\noindent {\bf Step \#1:} {\it Construct a reference grid.}
Since $\mathcal{X}$ is bounded, without loss of generality suppose $\mathcal{X} \cup \partial \mathcal{X} \subseteq [0,1]^d$.
For $M \in \mathbb{N}$, define a uniform grid of reference points $\{\bm{g}_i\}_{i=1}^G \subset [0,1]^d$ consisting of all $G = M^d$ states of the form $\bm{g} = (g_1,\dots,g_d)$ where 
$g_i \in \{0,\frac{1}{M-1},\dots,\frac{M-2}{M-1},1\}$ 
(Fig. \ref{fig:step1}).
We will require that this grid has a sufficiently fine resolution; specifically we suppose that 
$$
M \geq \frac{\sqrt{d}(1 + \sin\theta)}{2r\sin\theta} + 1, 
$$
where $r$, $\theta$ are defined by the interior cone condition that $\mathcal{X} \cup \partial \mathcal{X}$ is assumed to satisfy.

\begin{figure}[t!]
\centering
\begin{subfigure}[b]{0.46\textwidth}
\includegraphics[page=1,width=\textwidth,clip,trim = 2.5cm 3.5cm 17.5cm 1.6cm]{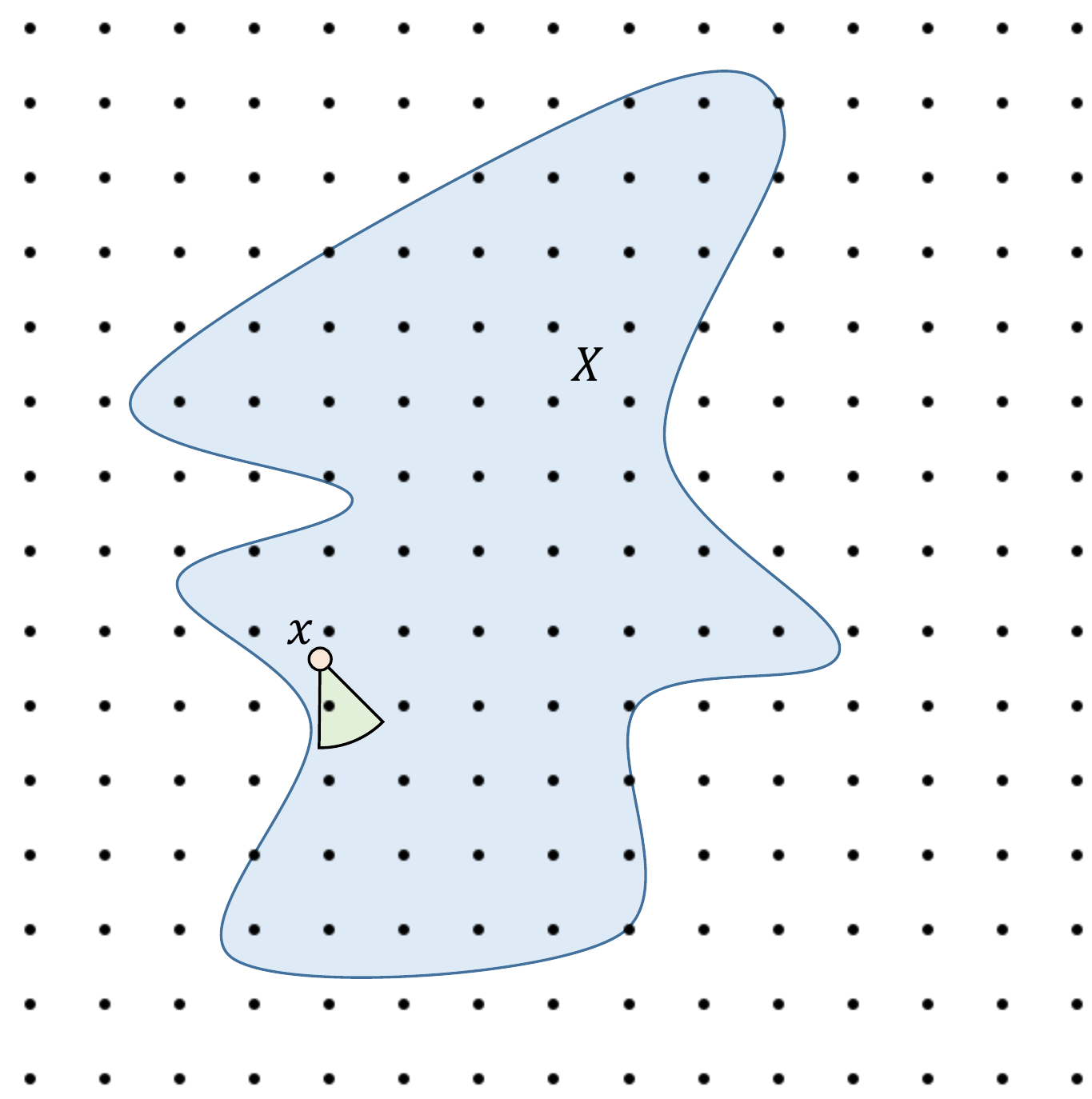}
\caption{Step \#1}
\label{fig:step1}
\end{subfigure}
\hspace{20pt}
\begin{subfigure}[b]{0.46\textwidth}
\includegraphics[page=2,width=\textwidth,clip,trim = 2.5cm 3.5cm 17.5cm 1.6cm]{figures/cartoon.pdf}
\caption{Step \#2}
\label{fig:step2}
\end{subfigure}
\caption{Schematic of the steps involved in proving Lemma \ref{tech lemma}.}
\end{figure}

\vspace{5pt}
\noindent {\bf Step \#2:} {\it Find a grid point $\bm{g}_i$ near to $\bm{x}$.}
Define  
$$
R = \frac{\sqrt{d}}{2} \left(\frac{1}{\sin\theta} + 1\right).
$$
We claim that there exists a grid point $\bm{g}_i$ such that $\|\bm{x} - \bm{g}_i\|_2 < \frac{R}{M-1}$.
Indeed, from the interior cone condition there exists a cone $\mathcal{C}(\bm{x},\bm{\xi},\theta,r) \subseteq \mathcal{X}$ (Fig. \ref{fig:step1}).
Define 
$$
h = \frac{\sqrt{d}}{2(M-1) \sin \theta}.
$$
The fine grid resolution implies $h \leq r / (1 + \sin\theta)$.
From \cite{Wendland}, Lemma 3.7, it follows that the cone $\mathcal{C}(\bm{x},\bm{\xi},\theta,r)$ contains the ball $B(\bm{y},\epsilon)$ with centre $\bm{y} = \bm{x} + h \bm{\xi}$ and radius $\epsilon = h \sin \theta$ (Fig. \ref{fig:step2}).
Now, the fill distance for the grid $\{\bm{g}_i\}_{i=1}^G$ relative to the space $[0,1]^d$ can be shown to equal $\sqrt{d} / (2(M-1))$, which is exactly equal to the radius $\epsilon$.
It follows that $B(\bm{y},\epsilon)$ must contain a grid point $\bm{g}_i$ for some $i \in \{1,\dots,G\}$.
From the triangle inequality;
\begin{eqnarray*}
\|\bm{x} - \bm{g}_i\|_2 & \leq & \underbrace{\|\bm{x} - \bm{y}\|_2}_{=h} + \underbrace{\|\bm{y} - \bm{g}_i\|_2}_{\leq \epsilon} \\
& = & \frac{\sqrt{d}}{2(M-1)\sin\theta} + \frac{\sqrt{d}}{2(M-1)} \; = \; \frac{R}{M-1}
\end{eqnarray*}
This establishes the claim.

\vspace{5pt}
\noindent {\bf Step \#3:} {\it A set of points that `cover the grid' must contain at least one point that is near to $\bm{x}$. }
Consider the event 
$$
E = \left[\forall i \exists j : \|\bm{g}_i - \bm{x}_j\|_2 \leq \frac{1}{M-1} \right].
$$ 
Conditional on $E$, there exists $\bm{x}_j$ such that $\|\bm{g}_i - \bm{x}_j\|_2 \leq \frac{1}{M-1}$.
It follows that, conditional on $E$, we have
\begin{eqnarray*}
\|\bm{x} - \bm{x}_j\|_2 & \leq & \|\bm{x} - \bm{g}_i\|_2 + \|\bm{g}_i - \bm{x}_j\|_2 \\
& \leq & \frac{R}{M-1} + \frac{1}{M-1} \; = \; \frac{R+1}{M-1}
\end{eqnarray*}
where we have used the result of Step \#1.
Thus the event $E$, or `covering the grid', implies the event $[h_{\mathcal{D}_0} \leq \frac{R+1}{M-1}]$ (Fig. \ref{fig:step3}).

\begin{figure}[t!]
\centering
\begin{subfigure}[b]{0.46\textwidth}
\includegraphics[page=3,width=\textwidth,clip,trim = 0.9cm 1.1cm 16cm 0.9cm]{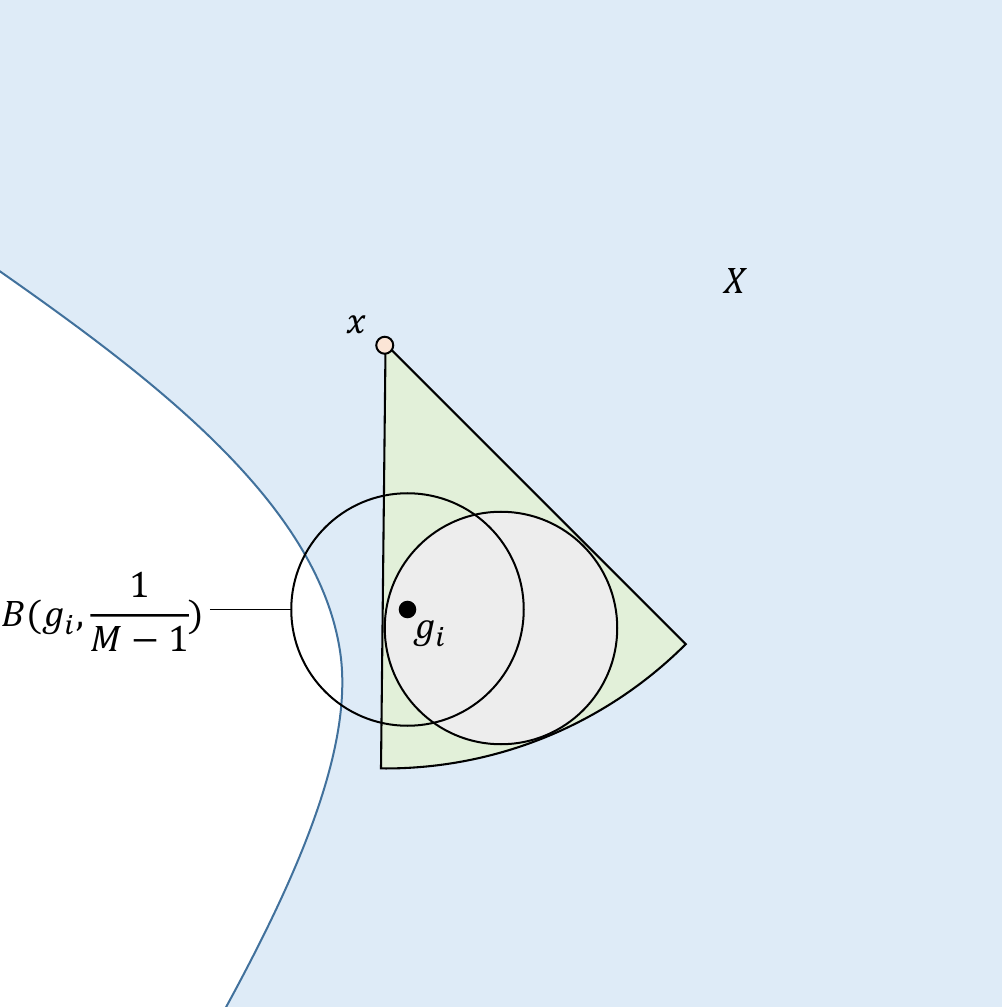}
\caption{Step \#3}
\label{fig:step3}
\end{subfigure}
\hspace{20pt}
\begin{subfigure}[b]{0.46\textwidth}
\includegraphics[page=4,width=\textwidth,clip,trim = 0.9cm 1.1cm 16cm 0.9cm]{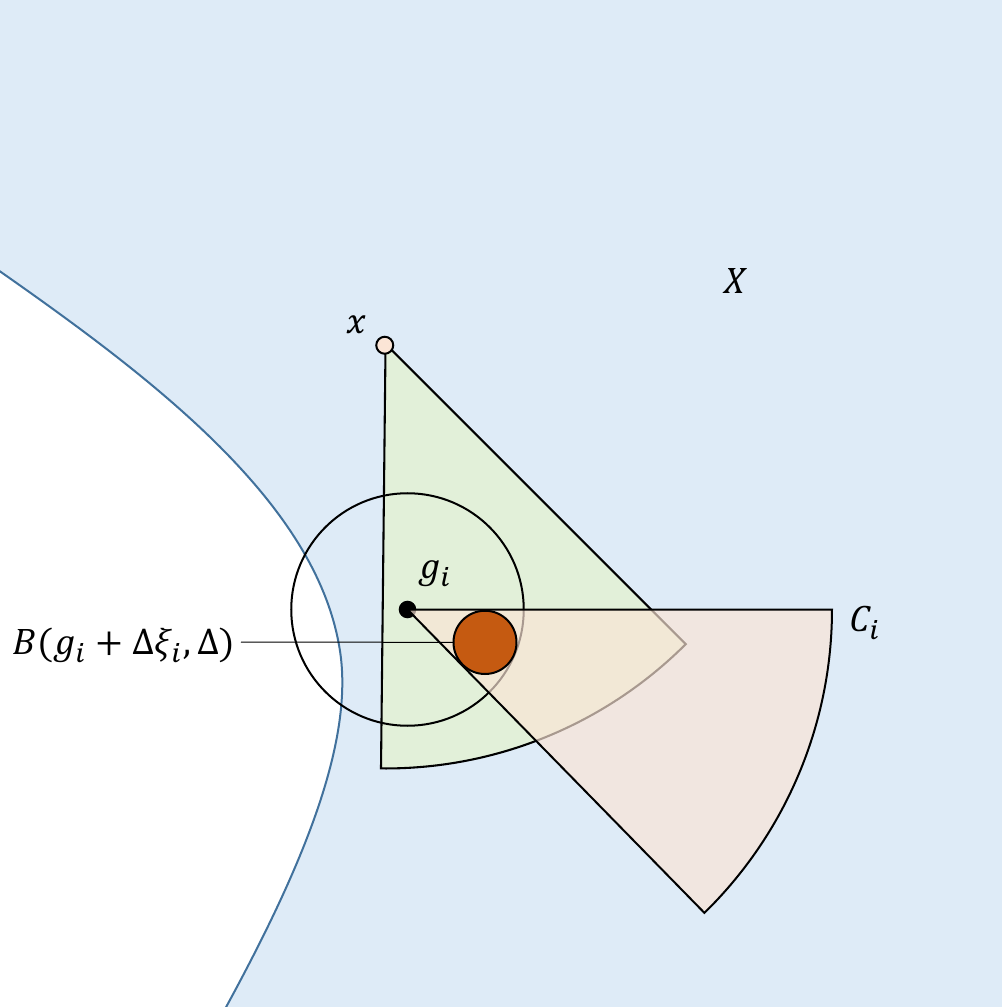}
\caption{Step \#4}
\label{fig:step4}
\end{subfigure}
\caption{Schematic of the steps involved in proving Lemma \ref{tech lemma}.}
\end{figure}

\vspace{5pt}
\noindent {\bf Step \#4:} {\it The grid is covered with high probability. }
Next, we upper-bound the probability $\mathbb{P}_{\mathcal{D}_0}[E^c]$ of the event $E^c$.
For this, note that for all $i \in \{1,\dots,G\}$ there exists a unit vector $\bm{\xi}_i$ such that the cone $\mathcal{C}_i = \mathcal{C}(\bm{g}_i,\bm{\xi}_i,\theta,r)$ is contained in $\mathcal{X}$ (Fig. \ref{fig:step4}).
It follows that
\begin{eqnarray*}
\text{vol}\left(B\left(\bm{g}_i,\frac{1}{M-1}\right) \cap \mathcal{X}\right) & \geq & \text{vol}\left(B\left(\bm{g}_i,\frac{1}{M-1}\right) \cap \mathcal{C}_i\right) \\
& = & \text{vol}\left(\mathcal{C}\left(\bm{g}_i,\bm{\xi}_i,\theta,\frac{1}{M-1}\right)\right)  
\end{eqnarray*}
where the final inequality follows since $\frac{1}{M-1} \leq r$ and the intersection of a cone with a ball is another cone.
Now, from \cite{Wendland}, Lemma 3.7, the ball with centre $\bm{g}_i + \Delta \bm{\xi}_i$ and radius 
$$
\Delta = \frac{\sin\theta}{(M-1)(1 + \sin\theta)}
$$
is contained in the cone $\mathcal{C}(\bm{g}_i,\bm{\xi}_i,\theta,\frac{1}{M-1})$.
Continuing,
\begin{eqnarray*}
\text{vol}\left(\mathcal{C}\left(\bm{g}_i,\bm{\xi}_i,\theta,\frac{1}{M-1}\right)\right)  & \geq & \text{vol}(B(\bm{g}_i + \Delta \bm{\xi}_i,\Delta)) \\
& = & V \; := \; \frac{(\frac{\sin\theta}{1 + \sin\theta})^d \pi^{d/2}}{\Gamma(\frac{d}{2} + 1) (M-1)^d}.
\end{eqnarray*}
Now, we have
\begin{eqnarray*}
\mathbb{P}_{\mathcal{D}_0}[E^c] & = & \mathbb{P}_{\mathcal{D}_0}\left[\exists i : \forall j, \|\bm{g}_i - \bm{x}_j \|_2 > \frac{1}{M-1}\right] \\
& \leq & \sum_{i=1}^G \underbrace{\mathbb{P}_{\mathcal{D}_0} \left[\forall j, \|\bm{g}_i - \bm{x}_j\|_2 > \frac{1}{M-1}\right]}_{(*)}.
\end{eqnarray*}
The probability $(*)$ can be bounded above using independence of the samples.
Indeed, the probability that a random draw from $\Pi$ lands in $B(\bm{g}_i,\frac{1}{M-1})$ can be lower-bounded by $V$ times $\eta = \inf_{\bm{x} \in \mathcal{X} \cup \partial \mathcal{X}} \pi(\bm{x}) > 0$ since, under (A1,5), $\pi$ can be decomposed as a mixture density with two components, the first uniform over $\mathcal{X} \cup \partial \mathcal{X}$ with mixture weight $\eta$ and the second a component with density $\tilde{\pi}(\bm{x}) \propto \pi(\bm{x}) - \eta$ and mixture weight $1 - \eta$.
Thus we have
\begin{eqnarray}
\mathbb{P}_{\mathcal{D}_0} \left[\forall j, \|\bm{g}_i - \bm{x}_j\|_2 > \frac{1}{M-1}\right] & \leq & ( 1 - V \eta )^m. \label{using independence}
\end{eqnarray}
Note that this is the only place in the proof that independence of the $\{\bm{x}_i\}$ is required.

\vspace{5pt}
\noindent {\bf Step \#5:} {\it The fill distance is small with high probability. }
The result of Step \#4 can be used to derive a lower bound on the distribution function of the fill distance:
\begin{eqnarray*}
\mathbb{P}_{\mathcal{D}_0}\left[ h_{\mathcal{D}_0} > \frac{R+1}{M-1}\right] \; \leq \; \mathbb{P}_{\mathcal{D}_0}[E^c] & \leq & G ( 1 - V \eta )^m \\
& = & G \left( 1 - \frac{ (\frac{\sin\theta}{1 + \sin\theta})^d \pi^{d/2} \eta}{ \Gamma(\frac{d}{2} + 1) (M-1)^d} \right)^m
\end{eqnarray*}
Letting $\zeta = \frac{R+1}{M-1}$ implies that $M = 1 + \frac{R+1}{\zeta}$ and $G = (1 + \frac{R+1}{\zeta})^d$.
In this parametrisation, when $M > R+2$ we have $0<\zeta<1$ and hence
\begin{eqnarray}
\mathbb{P}_{\mathcal{D}_0}[h_{\mathcal{D}_0} > \zeta] & \leq &  \left( 1 + \frac{R+1}{\zeta} \right)^d \left[ 1 - \frac{ (\frac{\sin\theta}{1 + \sin\theta})^d \pi^{d/2} \eta}{ \Gamma(\frac{d}{2}+1) (R+1)^d} \zeta^d \right]^m \nonumber \\
& \leq  & \left(\frac{R+2}{\zeta}\right)^d (1 - C_d \zeta^d)^m, \label{tech eq1}
\end{eqnarray}
where we have written 
$$
C_d = \frac{(\frac{\sin\theta}{1 + \sin\theta})^d \pi^{d/2} \eta }{ \Gamma(\frac{d}{2}+1)(R+1)^d}.
$$
While Eqn. \ref{tech eq1} holds only for $\zeta$ of the form $\frac{R+1}{M-1}$, it can be made to hold for all $0 < \zeta < 1$ by replacing $C_d$ with $\tilde{C}_d = 2^{-d} C_d$.
This is because for any $0 < \zeta < 1$ there exists $M >1$ such that $\tilde{\zeta} = \frac{R+1}{M-1}$ satisfies $\frac{\zeta}{2} \leq \tilde{\zeta} < \zeta$, along with the fact that $\mathbb{P}_{\mathcal{D}_0}[h_{\mathcal{D}_0} > \zeta] \leq \mathbb{P}_{\mathcal{D}_0}[h_{\mathcal{D}_0} > \tilde{\zeta}]$.

\vspace{5pt}
\noindent {\bf Step \#6:} {\it Putting it all together. }
From $\mathcal{X} \subseteq [0,1]^d$ we have that $h_{\mathcal{D}_0} \in [0,1]$ and hence, since $g$ is continuous and $[0,1]$ is compact, $\sup_{h \in [0,1]} g(h) < \infty$.
Without loss of generality, and with probability one, we have $g(h_{\mathcal{D}_0}) \leq 1$.
(This is without loss of generality since we can simply redefine $g$ as $g / g(1)$.)
From the reverse Markov inequality, for all $\zeta < \mathbb{E}_{\mathcal{D}_0}[g(h_{\mathcal{D}_0})]$,
\begin{eqnarray*}
\mathbb{P}_{\mathcal{D}_0}[g(h_{\mathcal{D}_0}) > \zeta] & \geq & \frac{ \mathbb{E}_{\mathcal{D}_0}[g(h_{\mathcal{D}_0})] - \zeta}{1 - \zeta} 
\end{eqnarray*}
and upon rearranging
\begin{eqnarray}
\mathbb{E}_{\mathcal{D}_0}[g(h_{\mathcal{D}_0})] & \leq & \zeta + (1 - \zeta) \mathbb{P}_{\mathcal{D}_0}[g(h_{\mathcal{D}_0}) > \zeta]. \label{tech eq2}
\end{eqnarray}
Since $g$ is continuous and monotone, $g^{-1}$ exists and we have $\mathbb{P}_{\mathcal{D}_0}[g(h_{\mathcal{D}_0}) \leq \zeta] = \mathbb{P}_{\mathcal{D}_0}[h_{\mathcal{D}_0} \leq g^{-1}(\zeta)]$.
This allows us to combine Eqns. \ref{tech eq1} and \ref{tech eq2}, obtaining
\begin{eqnarray*}
\mathbb{E}_{\mathcal{D}_0}[g(h_{\mathcal{D}_0})] \quad & \leq & \zeta + (1 - \zeta)\left( \frac{R+2}{g^{-1}(\zeta)} \right)^d (1 - \tilde{C}_d(g^{-1}(\zeta))^d)^m \\
& \leq & \zeta + \left( \frac{R+2}{g^{-1}(\zeta)} \right)^d (1 - \tilde{C}_d(g^{-1}(\zeta))^d)^m.
\end{eqnarray*}
Now, letting $\zeta = g(m^{-\delta})$ for some fixed $\delta$, subject to $\frac{1}{3d} \leq \delta < \frac{1}{d}$, and varying $m$, we have that
\begin{eqnarray}
\mathbb{E}_{\mathcal{D}_0}[g(h_{\mathcal{D}_0})] & \leq & g(m^{-\delta}) + (R+2)^d \underbrace{m^{d\delta} (1 - \tilde{C}_d m^{-d \delta})^m}_{(**)} \label{tech eq3}
\end{eqnarray}
where $(**) = O( m^{d\delta} \exp(- \tilde{C}_d m^{1 - d \delta}) )$.
Indeed, since $\log(1-x) \leq -x$ for all $|x|<1$,
\begin{eqnarray*}
\log[ m^{d\delta} (1 - \tilde{C}_d m^{-d \delta})^m ] & = & d\delta \log(m) + m \log(1 - \tilde{C}_d m^{-d\delta}) \\
& \leq & d\delta \log(m) - \tilde{C}_d m^{1 - d\delta} \; = \; \log[ m^{d\delta} \exp( -\tilde{C}_d m^{1 - d\delta} ) ].
\end{eqnarray*}
Writing $x = m^{-\delta}$,
\begin{eqnarray*}
\frac{ m^{d\delta} \exp(-\tilde{C}_d m^{1 - d\delta}) }{ g(m^{-\delta}) } & = & \frac{ x^{-d} \exp(-\tilde{C}_d x^{d - 1/\delta}) }{ g(x) } \\
& \leq & \frac{ x^{-d} \exp(-\tilde{C}_d x^{-2d}) }{ g(x) } \; \; \; \text{(take } \delta = \frac{1}{3d} \text{).}
\end{eqnarray*}
Under the hypothesis on the limiting behavior of $g(x)$ as $x \downarrow 0$, we have that
\begin{eqnarray*}
\lim_{m \rightarrow \infty} \frac{ m^{d\delta} \exp(-\tilde{C}_d m^{1 - d\delta}) }{ g(m^{-\delta}) } & = &  \lim_{x \downarrow 0} \frac{ x^{-d} \exp(-\tilde{C}_d x^{-2d}) }{ g(x) } \\
& \leq & \lim_{x \downarrow 0} \frac{ x^{-d} \exp(- x^{-3d}) }{ g(x) }  \; = \; 0
\end{eqnarray*}
Thus the right hand side of Eqn. \ref{tech eq3} is asymptotically minimised by taking $\delta$ as large as possible, subject to $(**)$ converging exponentially fast, i.e. $\delta = \frac{1}{d} - \epsilon$ for $\epsilon > 0$ arbitrarily small.
We therefore conclude that $\mathbb{E}_{\mathcal{D}_0}[g(h_{\mathcal{D}_0})] = O(g(m^{-1 / d + \epsilon}))$, as required.
\end{proof}

\begin{proof}[Proof of Theorem \ref{independent}]
Unbiasedness follows directly from the structure of splitting estimators (Eqn. \ref{splitting estimators}).
From Eqn. \ref{analyse mse} it suffices to consider the rate $\delta$ at which $\sigma^2(f - f_m)$ vanishes.
First, note that 
\begin{eqnarray*}
\sigma^2(f - f_m) & = & \int \left[f(\bm{x}) - f_m(\bm{x}) - \int [f(\bm{x}') - f_m(\bm{x}')] \Pi(\d{}\bm{x}')  \right]^2 \Pi(\d{} \bm{x}) \\ 
& = & \int [f(\bm{x}) - f_m(\bm{x})]^2 \Pi(\d{} \bm{x}) - \left[\int [f(\bm{x}') - f_m(\bm{x}')] \Pi(\d{} \bm{x}')  \right]^2 \\ 
& \leq & \int [f(\bm{x}) - f_m(\bm{x})]^2 \Pi(\d{} \bm{x}).
\end{eqnarray*}
Second, observe that from (A1-3) the kernel $k_+ \in C_2^{a \wedge b}(\mathcal{X} \cup \partial \mathcal{X})$.
Thus from Theorem 11.13 of \cite{Wendland} there exists $h > 0$, $C>0$ such that, whenever $h_{\mathcal{D}_0} < h$, we have
\begin{eqnarray*}
|f(\bm{x}) - f_m(\bm{x})| & \leq & C h_{\mathcal{D}_0}^{a \wedge b} \|f\|_{\mathcal{H}_+} 
\end{eqnarray*}
for all $\bm{x} \in \mathcal{X}$.
An unconditional bound, after squaring and integrating according to $\Pi$, is
\begin{eqnarray*}
1_{h_{\mathcal{D}_0} < h} \int [f(\bm{x}) - f_m(\bm{x})]^2 \Pi(\d{} \bm{x}) & \leq & C^2 h_{\mathcal{D}_0}^{2(a \wedge b)} \|f\|_{\mathcal{H}_+}^2.
\end{eqnarray*}
Third, combining the first two parts and taking an expectation over the $m$ samples in $\mathcal{D}_0$ gives 
\begin{eqnarray*}
\mathbb{E}_{\mathcal{D}_0}[1_{h_{\mathcal{D}_0} < h} \sigma^2(f - f_m)] & \leq & C^2 \|f\|_{\mathcal{H}_+}^2 \mathbb{E}_{\mathcal{D}_0}[h_{\mathcal{D}_0}^{2(a \wedge b)}].
\end{eqnarray*}
Finally, from Lemma \ref{tech lemma} with $g(x) = x^{2(a \wedge b)}$ we have that $\mathbb{E}_{\mathcal{D}_0}[h_{\mathcal{D}_0}^{2(a \wedge b)}] = O(m^{-2 \frac{a \wedge b}{d} + \epsilon})$ for $\epsilon > 0$ arbitrarily small.
The result follows from Eqn. \ref{analyse mse}.
\end{proof}

\begin{proof}[Proof of Lemma \ref{tech lemma 2}]
In the proof of Lemma \ref{tech lemma}, independence of the samples $\{\bm{x}_i\}_{i=1}^n$ was only used to establish Eqn. \ref{using independence}.
Below we derive an almost equivalent inequality that is valid for non-independent samples arising from the Markov chain sample path.

From \citet[][Prop. 1]{Roberts2} a uniformly ergodic, reversible Markov chain is strongly uniformly ergodic; i.e. there exists $N \in \mathbb{N}$ and $0 < \upsilon < 1$ such that for all $\bm{x} \in \mathcal{X}$ the $N$-step transition density $P^N(\bm{x},\cdot)$ satisfies the minorisation condition $P^N(\bm{x},\cdot) \geq \upsilon \Pi[\cdot]$.
Manipulating this minorisation condition gives $P^N(\bm{x},A^c) = 1 - P^N(\bm{x},A) \leq 1 - \upsilon \Pi(A)$ for $A \in \mathcal{A}$.

Fix $A = B(\bm{g}_i , \frac{1}{M-1})$.
Denote the initial distribution of the Markov chain by $\Pi_1$ and, for notational simplicity only, assume $\Pi_1$ has a density $\pi_1 = \d\Pi_1 / \d\Lambda$.
Define the hitting time $\tau_A := \min\{n : \bm{x}_n \in A \text{ given } \bm{x}_1 \sim \Pi_1\}$.
Then for $n \geq N$ we have the following bound:
\begin{eqnarray*}
\mathbb{P}[\tau_A > n] & = & \int_{\bm{x}_1 \in A^c} \dots \int_{\bm{x}_n \in A^c} \pi_1(\bm{x}_1) P(\bm{x}_1,\bm{x}_2) \dots P(\bm{x}_{n-1},\bm{x}_n) \d{\bm{x}}_n \dots \d{\bm{x}}_1 \nonumber \\
& = & \int_{\bm{x}_1 \in A^c} \dots \int_{\bm{x}_{n-N} \in A^c} \pi_1(\bm{x}_1) P(\bm{x}_1,\bm{x}_2) \dots P(\bm{x}_{n-N-1},\bm{x}_{n-N}) \\
& & \hspace{-55pt} \times \underbrace{\int_{\bm{x}_{n-N+1} \in A^c} \hspace{-10pt} \dots \int_{\bm{x}_n \in A^c} \begin{array}{l} \hspace{-20pt} P(\bm{x}_{n-N},\bm{x}_{n-N+1}) \dots P(\bm{x}_{n-1},\bm{x}_n) \\ \hspace{100pt} \d{\bm{x}}_{n-N+1} \d{\bm{x}}_n \end{array} }_{ \leq P^N(\bm{x}_{n-N},A^c)} \d{\bm{x}}_{n-N} \dots \d{\bm{x}}_1 \\
& \leq & [1 - \upsilon \Pi(A)] \\
& & \times \int_{\bm{x}_1 \in A^c} \dots \int_{\bm{x}_{n-N} \in A^c} \begin{array}{l} \hspace{-20pt} \pi_1(\bm{x}_1) P(\bm{x}_1,\bm{x}_2) \dots P(\bm{x}_{n-N-1},\bm{x}_{n-N}) \\ \hspace{110pt} \d{\bm{x}}_{n-N} \dots \d{\bm{x}}_1 \end{array}  \nonumber \\
\dots & \leq & [1 - \upsilon \Pi(A)]^{\lfloor n/N \rfloor} \\
& & \hspace{-55pt} \times \int_{\bm{x}_1 \in A^c} \dots \int_{\bm{x}_{n-\lfloor n / N \rfloor N} \in A^c} \begin{array}{l} \hspace{-30pt} \pi_1(\bm{x}_1) P(\bm{x}_1,\bm{x}_2) \dots P(\bm{x}_{n-\lfloor n / N \rfloor N -1},\bm{x}_{n-\lfloor n / N \rfloor N}) \\ \hspace{120pt} \d{\bm{x}}_{n-\lfloor n / N \rfloor N} \dots \d{\bm{x}}_1 \end{array} \nonumber \\
& \leq & [1 - \upsilon \Pi(A)]^{\lfloor n/N \rfloor} \label{dependent bounding}
\end{eqnarray*}
Employing this bound, we have
\begin{eqnarray*}
\mathbb{P}_{\mathcal{D}_0} \left[\forall j, \|\bm{g}_i - \bm{x}_j\|_2 > \frac{1}{M-1}\right] & = & \mathbb{P}_{\mathcal{D}_0}[\tau_A > m] \; \leq \; [1 - \upsilon \Pi(A)]^{\lfloor m / N \rfloor}.
\end{eqnarray*}
As before we have $\Pi(A) \geq V \eta$ and hence
\begin{eqnarray*}
\mathbb{P}_{\mathcal{D}_0} \left[\forall j, \|\bm{g}_i - \bm{x}_j\|_2 > \frac{1}{M-1}\right] & \leq &  ( 1 - \upsilon V \eta )^{\lfloor m / N \rfloor}. \label{generalised eqn}
\end{eqnarray*}
Eqn. \ref{generalised eqn} is essentially identical to Eqn. \ref{using independence} up to the inclusion of a factor $0<\upsilon<1$ and a factor $1/N$.
Arguing as in Lemma \ref{tech lemma} with $\tilde{C}_d$ replaced by $\upsilon N^{-1} \times \tilde{C}_d$ completes the proof.
\end{proof}

\begin{proof}[Proof of Theorem \ref{dependent}]
We appeal to Theorem 1 of \cite{Roberts}.
This states that reversible, geometrically ergodic Markov chains $\{\bm{x}_i\}_{i=1}^n$ are {\it variance bounding}, meaning that there exists $K < \infty$ such that
\begin{eqnarray}
\lim_{n \rightarrow \infty} \frac{1}{n} \mathbb{V} \left[ \sum_{i=1}^n h(\bm{x}_i) \right] \leq K \sigma^2(h)
\label{variance bounding}
\end{eqnarray}
holds for all $h \in L^2(\mathcal{X},\Pi)$.
Suppose that the chain starts at stationarity, so that Eqn. \ref{variance bounding} is equivalent to the statement
\begin{eqnarray}
\lim_{n \rightarrow \infty} n \mathbb{E} \left[ \left( \frac{1}{n} \sum_{i=1}^n h(\bm{x}_i) - \int h \d\Pi \right)^2 \right] \leq K \sigma^2(h). \label{equivalent statement}
\end{eqnarray}
(An arbitrary initial condition $\bm{x}_1 \sim \Pi_1$ can be handled using standard renewal theory; we do not present the details here.)
Applying Eqn. \ref{equivalent statement} to the control functional estimator produces
\begin{eqnarray} \label{T2bound1}
\lim_{n \rightarrow \infty} (n-m) \mathbb{E}_{\mathcal{D}_1} \left[\left(I_{m,n} - \int f \d\Pi\right)^2\right] & \leq & K \sigma^2(f - f_m).
\end{eqnarray}
Arguing as in the proof of Theorem \ref{independent}, we have 
\begin{eqnarray} \label{T2bound2}
\mathbb{E}_{\mathcal{D}_0} [1_{h_{\mathcal{D}_0} < h} \sigma^2(f - f_m)] & \leq & C^2 \|f\|_{\mathcal{H}_+}^2  \mathbb{E}_{\mathcal{D}_0}[h_{\mathcal{D}_0}^{2(a \wedge b)}].
\end{eqnarray}
Finally, it remains only to show that the scaling relation $\mathbb{E}_{\mathcal{D}_0}[h_{\mathcal{D}_0}^{2(a \wedge b)}] = O(n^{-2(a \wedge b) / d + \epsilon})$, where $\epsilon > 0$ can be arbitrarily small, holds in the non-independent setting.
This was precisely the content of Lemma \ref{tech lemma 2}.
\end{proof}

Denote $\bm{k} : \mathcal{X} \times \mathcal{X} \rightarrow \mathbb{R}^d$ for the vector $\bm{k}(\bm{x},\bm{x}') = \bm{1} k(\bm{x},\bm{x}')$.
The shorthand notation $\mathbb{S}_q[\bm{k}(\bm{x},\bm{x}')]$ will be used to refer to $\mathbb{S}_q[\bm{k}(\cdot,\bm{x}')](\bm{x})$; i.e. the action of the Stein operator on the first argument of a bivariate function.

\begin{lemma} \label{new lemma}
Assume (A3).
For each $q \in \mathcal{Q}(k) \cap \mathcal{R}(k)$ it holds that 
\begin{eqnarray}
\int \mathbb{S}_q[\bm{k}(\bm{x},\bm{x}')] Q(\d{} \bm{x}) & = & 0, \label{eqn: stein on kernel}
\end{eqnarray}
where the left hand side exists for all $\bm{x}' \in \mathcal{X}$.
\end{lemma}
\begin{proof}
This is a generalisation of Lemma \ref{lem integrate to zero}.
From (A3) and $q \in \mathcal{Q}(k)$, the function $\mathbb{S}_q[\bm{k}(\cdot,\bm{x}')]$ exists on $\mathcal{X}$ for each $\bm{x}' \in \mathcal{X}$.
In particular $\mathbb{S}_q[\bm{k}(\cdot,\bm{x}')] \in L^2(\mathcal{X} \cup \partial \mathcal{X},Q)$, so that the left hand side of Eqn. \ref{eqn: stein on kernel} exists.
Then
\begin{eqnarray*}
\int \mathbb{S}_q[\bm{k}(\bm{x},\bm{x}')] Q(\d{} \bm{x}) & = & \int [ \nabla_{\bm{x}} \cdot \bm{k}(\bm{x},\bm{x}') + \bm{k}(\bm{x},\bm{x}') \cdot \nabla_{\bm{x}} \log q(\bm{x}) ] \; Q(\d{}\bm{x}) \\
& = & \int [ q(\bm{x}) \nabla_{\bm{x}} \cdot \bm{k}(\bm{x},\bm{x}') + \bm{k}(\bm{x},\bm{x}') \cdot \nabla_{\bm{x}} q(\bm{x}) ] \d{}\bm{x} \\
& = & \int \nabla_{\bm{x}} \cdot \{ q(\bm{x}) \bm{k}(\bm{x},\bm{x}') \} \d{}\bm{x} \\
& \stackrel{(*)}{=} & \oint_{\partial \mathcal{X}} \underbrace{q(\bm{x}) \bm{k}(\bm{x},\bm{x}')}_{(**)} \cdot \bm{n}(\bm{x}) \; S(\d{} \bm{x}) \quad  = \quad 0,
\end{eqnarray*}
where $(*)$ is integration by parts and $(**)$  equals zero for all $\bm{x}' \in \partial \mathcal{X}$ since $q \in \mathcal{R}(k)$.
\end{proof}

\begin{proof}[Proof of Lemma \ref{characteristic}]
Recall that, for compact $\mathcal{X} \cup \partial \mathcal{X}$, the notion of $c$-universality is equivalent to \emph{$cc$-universality}, where the (weaker) topology of compact convergence is used in place of the $\|\cdot\|_\infty$ norm topology \citep[][Defn. 4.1]{Carmeli}.

For densities $q,q'$ on $(\mathcal{X} \cup \partial \mathcal{X},\mathcal{B})$ with $q = \d{} Q / \d{} \Lambda$, $q' = \d{} Q' / \d{} \Lambda$, define
\begin{eqnarray*}
T(q,q',k) & := & \left\| \int \mathbb{S}_{q'} [ \bm{k}(\bm{x},\cdot) ] Q(\mathrm{d}\bm{x}) - \int \mathbb{S}_{q'} [ \bm{k}(\bm{x},\cdot) ] Q'(\mathrm{d}\bm{x}) \right\|_{\mathcal{H}},
\end{eqnarray*}
whenever the right hand side exists.
Observe, as in Lemma \ref{new lemma}, that $T(q,q',k)$ is well-defined for all $q,q' \in \mathcal{Q}(k)$.

This proof then proceeds in two steps:

\noindent {\bf Step \#1:}
First, we follow the proof of \citet[][Thm. 2.1]{Chwialkowski} in order to establish that
$$
q,q' \in \mathcal{Q}(k) \cap \mathcal{R}(k), \quad T(q,q',k) = 0 \quad \implies \quad q = q'.
$$
Indeed:
\begin{eqnarray*}
T(q,q',k) & = & \Bigg\| \int \mathbb{S}_{q'} [ \bm{k}(\bm{x},\cdot) ] Q(\mathrm{d}\bm{x})  - \underbrace{\int \mathbb{S}_{q'} [ \bm{k}(\bm{x},\cdot) ] Q'(\mathrm{d}\bm{x}) }_{ = 0 \text{ from Lemma \ref{new lemma}} } \Bigg\|_{\mathcal{H}} \\
& = & \left\| \int \bigg[ \sum_{i=1}^d \nabla_{x_i} k(\bm{x},\cdot) + k(\bm{x},\cdot) \nabla_{x_i} \log q'(\bm{x}) \bigg] Q(\mathrm{d}\bm{x}) \right\|_{\mathcal{H}} \\
& = & \Bigg\| \underbrace{ \int \bigg[ \sum_{i=1}^d \nabla_{x_i} k(\bm{x},\cdot) + k(\bm{x},\cdot) \nabla_{x_i} \log q(\bm{x}) \bigg] Q(\mathrm{d}\bm{x}) }_{ (*) } \\
& & \quad + \underbrace{\int \bigg[ \sum_{i=1}^d k(\bm{x},\cdot) \nabla_{x_i} \log q'(\bm{x}) - k(\bm{x},\cdot) \nabla_{x_i} \log q(\bm{x}) \bigg] Q(\mathrm{d}\bm{x}) }_{ (**) } \Bigg\|_{\mathcal{H}} 
\end{eqnarray*}
where 
$$
(*) = \int \mathbb{S}_q[\bm{k}(\bm{x},\cdot)] Q(\mathrm{d}\bm{x}),
$$ 
which is the zero function from Lemma \ref{new lemma} and $(**)$ is the mean embedding \citep{Smola} of the function
$$
g(\bm{x}) := \sum_{i=1}^d \nabla_{x_i} \log \frac{q'(\bm{x})}{q(\bm{x})}.
$$
Note that $g \in L^2(\mathcal{X} \cup \partial \mathcal{X},Q)$ follows from $q,q' \in \mathcal{Q}(k)$.
By assumption this embedding satisfies 
$$
\left\| \int \bigg[ \sum_{i=1}^d \nabla_{x_i} \log \frac{q'(\bm{x})}{q(\bm{x})} \bigg] k(\bm{x},\cdot) Q(\d{}\bm{x}) \right\|_{\mathcal{H}} = 0.
$$
Since $\mathcal{H}$ is $cc$-universal and $g \in L^2(\mathcal{X} \cup \partial \mathcal{X},\Pi)$, the embedding is zero if and only if $g = 0$ \citep[][Thm. 4.4c]{Carmeli}.
Thus
$$
\sum_{i=1}^d \nabla_{x_i} \log \frac{q'(\bm{x})}{q(\bm{x})} = 0.
$$
This implies that $q'$ and $q$ are proportional on $\mathcal{X}$ and, since $q,q'$ are both densities in $C^1(\mathcal{X} \cup \partial \mathcal{X})$, we must conclude $q=q'$.

\noindent {\bf Step \#2:}
It is sufficient to prove that $\mathcal{H}_+$ is dense in $(C^1(\mathcal{X} \cup \partial \mathcal{X}) , \|\cdot\|_2)$, since this set is dense in $(L^2(\mathcal{X} \cup \partial \mathcal{X},\Pi) , \|\cdot\|_2)$.
Now, suppose $\mathcal{H}_+$ is not dense in $(C^1(\mathcal{X} \cup \partial \mathcal{X}) , \|\cdot\|_2)$.
Then there exists $0 \neq f \in C^1(\mathcal{X} \cup \partial \mathcal{X})$ such that
$$
\int f \; \mathrm{d}\Pi = 0, \quad \quad \int f \psi \; \mathrm{d}\Pi = 0 \quad \forall \; \psi \in \mathcal{H}_0,
$$
the second requirement representing orthogonality of $f$ with respect to $\mathcal{H}_0$.
Compactness of $\mathcal{X} \cup \partial \mathcal{X}$ implies $f \in L^\infty(\mathcal{X} \cup \partial \mathcal{X})$.
Let
$$
Q := \frac{c + f}{c} \; \Pi, \quad Q' := \Pi, \quad c := 1 + \|f\|_\infty,
$$
so that $Q,Q'$ are both distributions with $Q \neq Q'$.
Moreover, under (A$\bar{2}$,4), both $Q,Q'$ admit densities $q,q'$ such that $q,q' \in \mathcal{Q}(k) \cap \mathcal{R}(k)$. 
Indeed, $q \in \mathcal{Q}(k)$ since
\begin{enumerate}
\item[(a)] $q \in C^1(\mathcal{X} \cup \partial \mathcal{X})$
\item[(b)] $q > 0$ on $\mathcal{X}$;
\item[(c)] $\nabla_{x_i} \log q = \frac{\nabla_{x_i} f}{c + f} + \nabla_{x_i} \log \pi \in L^2(\mathcal{X} \cup \partial \mathcal{X},Q')$ for all distributions $Q'$ on $(\mathcal{X} \cup \partial \mathcal{X},\mathcal{B})$, since $\nabla_{x_i} f \in C^0(\mathcal{X} \cup \partial \mathcal{X})$ and $c + f \geq 1$ on $\mathcal{X}$;
\end{enumerate}
and $q \in \mathcal{R}(k)$ since $q(\bm{x}) k(\bm{x},\bm{x}') = \frac{c + f(\bm{x})}{c} \pi(\bm{x}) k(\bm{x},\bm{x}')$, where $\frac{c + f(\bm{x})}{c}$ is bounded on $\mathcal{X}$ and $\pi(\bm{x}) k(\bm{x},\bm{x}') = 0$ for $\bm{x} \in \partial\mathcal{X}$ and $\bm{x}' \in \mathcal{X} \cup \partial \mathcal{X}$.

The purpose of this construction becomes clear from plugging this choice of $q,q'$ into the operator $T$:
\begin{eqnarray*}
T(q,q',k) & = & \left\| \int \mathbb{S}_{q'} [ \bm{k}(\bm{x},\cdot) ] Q(\mathrm{d}\bm{x}) - \int \mathbb{S}_{q'} [ \bm{k}(\bm{x},\cdot) ] Q'(\mathrm{d}\bm{x}) \right\|_{\mathcal{H}} \\
& = & \Bigg\| \frac{1}{c} \int \underbrace{ \mathbb{S}_{q'} [ \bm{k}(\bm{x},\cdot) ] }_{ (***) } f(\bm{x}) \Pi(\d{} \bm{x}) \Bigg\|_{\mathcal{H}} .
\end{eqnarray*}
The function $(***)$ belongs to $\mathcal{H}_0$; the definition of $f$ then implies that this integral is zero (orthogonality with respect to $\mathcal{H}_0$) and so we conclude that $T(q,q',k) = 0$.
From Step \#1 we then conclude that $q = q'$.
This is a contradiction and so $\mathcal{H}_+$ has been shown to be dense in $(C^1(\mathcal{X} \cup \partial \mathcal{X}) , \|\cdot\|_2)$.
\end{proof}

\subsection{Formulae for Wendland Kernels} \label{wendland formulae}

\begin{table}[t!]
\footnotesize
\begin{tabular}{|c|c|p{12cm}|} \hline
Smoothness & Function & Expression \\ \hline
$b=0$ & $\varphi(z)$ & $[(\ell + 1)z + 1] (1-z)_+^{\ell + 1}$ \\
(i.e. $k \in C_2^1$) & $\varphi^{(1)}(z)$ & $- (\ell^2 + 3 \ell + 2) z (1-z)_+^\ell$ \\
& $\varphi^{(2)}(z)$ & $(\ell^2 + 3\ell + 2) [(\ell + 1) z - 1] (1-z)_+^{\ell - 1}$ \\ \hline
$b=1$ & $\varphi(z)$ & $\frac{1}{3} [ (\ell^2 + 4 \ell + 3) z^2 + 3(\ell + 2) z + 3 ] (1-z)_+^{\ell + 2}$ \\
(i.e. $k \in C_2^2$) & $\varphi^{(1)}(z)$ & $- \frac{1}{3} (\ell^2 + 7 \ell + 12) z [(\ell + 1)z + 1] (1-z)_+^{\ell + 1}$ \\
& $\varphi^{(2)}(z)$ & $\frac{1}{3}(\ell^2 + 7 \ell + 12) (\ell^2 z^2 + 4 \ell z^2 - \ell z + 3z^2 - 1) (1-z)_+^\ell$ \\ \hline
$b=2$ & $\varphi(z)$ & $\frac{1}{15}[(\ell^3+9\ell^2+23\ell + 15) z^3 + (6\ell^2 + 36\ell + 45)z^2 + 15(\ell+3)z + 15] (1-z)_+^{\ell + 3}$ \\
(i.e. $k \in C_2^3$) & $\varphi^{(1)}(z)$ & $- \frac{1}{15} (\ell^2 + 11\ell + 30) (\ell^2 z^2 + 4 \ell z^2 + 3 \ell z + 3 z^2 + 6z + 3) z (1-z)_+^{\ell + 2}$ \\
& $\varphi^{(2)}(z)$ & $\frac{1}{15} (\ell^2 + 11\ell + 30) (\ell^3 z^3 + 9 \ell^2 z^3 + 23\ell z^3 + 6 \ell z^2 - 3 \ell z + 15 z^3 + 15z^2 - 3z - 3) (1-z)_+^{\ell + 1}$ \\ \hline
\end{tabular}

\caption{Wendland's compact support functions and their first and second derivatives.
Here $\ell = \lfloor \frac{d}{2} + b + 2 \rfloor$.
Note that $z_+^k$ should be interpreted as $(z_+)^k$ rather than $(z^k)_+$.
The associated native spaces are equivalent to $H^{\frac{d}{2} + b + \frac{3}{2}}(\mathbb{R}^d)$.}
\label{tab:Wendland}
\end{table}
 
In this final section, explicit formulae for derivatives of the compact support radial functions of \cite{Wendland2} are provided in Table \ref{tab:Wendland}.
To derive the kernel $k_+$, let $r = \|\bm{x} - \bm{x}'\|_2$.
Then the identities
\begin{eqnarray*}
\nabla_{\bm{x}} \tilde{k} \; = \; - \nabla_{\bm{x}'} \tilde{k} & = & \frac{\bm{x} - \bm{x}'}{hr} \varphi^{(1)} \left(\frac{r}{h}\right) \\
\nabla_{\bm{x}} \cdot \nabla_{\bm{x}'} \tilde{k} & = & - \frac{1}{h^2} \varphi^{(2)} \left(\frac{r}{h}\right) 
\end{eqnarray*}
and
\begin{eqnarray*}
\nabla_{\bm{x}} k(\bm{x},\bm{x}') & = & [\nabla_{\bm{x}} \delta(\bm{x})] \delta(\bm{x}') \tilde{k}(\bm{x},\bm{x}') + \delta(\bm{x}) \delta(\bm{x}') \nabla_{\bm{x}} \tilde{k}(\bm{x},\bm{x}') \\
\nabla_{\bm{x}'} k(\bm{x},\bm{x}') & = & \delta(\bm{x}) [\nabla_{\bm{x}'} \delta(\bm{x}')] \tilde{k}(\bm{x},\bm{x}') + \delta(\bm{x}) \delta(\bm{x}') \nabla_{\bm{x}'} \tilde{k}(\bm{x},\bm{x}') \\
\nabla_{\bm{x}} \cdot \nabla_{\bm{x}'} k(\bm{x} , \bm{x}') & = & [\nabla_{\bm{x}} \delta(\bm{x})] \cdot [\nabla_{\bm{x}'} \delta(\bm{x}')] \tilde{k}(\bm{x},\bm{x}') + \delta(\bm{x}') [\nabla_{\bm{x}} \delta(\bm{x})] \cdot [\nabla_{\bm{x}'} \tilde{k}(\bm{x},\bm{x}')] \\
& & \quad + \delta(\bm{x}) [\nabla_{\bm{x}'} \delta(\bm{x}')] \cdot [\nabla_{\bm{x}} \tilde{k}(\bm{x},\bm{x}')] + \delta(\bm{x}) \delta(\bm{x}') \nabla_{\bm{x}} \cdot \nabla_{\bm{x}'} \tilde{k}(\bm{x},\bm{x}')
\end{eqnarray*}
can be used to obtain the kernel $k_+ = c + k_0$ in closed-form, in combination with Eqn. \ref{k0 expression}.

\end{document}